\documentclass{amsart}

\usepackage{setspace}
\usepackage{amsmath, amssymb,amsthm, amsfonts}
\usepackage{enumitem}
\newtheorem{thm}{Theorem}
\newtheorem{lma}{Lemma}
\newtheorem{prop}{Proposition}
\newtheorem{cor}{Corollary}

\newtheorem*{thm*}{Theorem}
\theoremstyle{definition}
\newtheorem{definition}{Definition}

\theoremstyle{remark}
\newtheorem{remark}{Remark}

\newcommand{\Ric}{\mbox{Ric}}
\newcommand{\R}{\mathbb R}
\newcommand{\Rm}{\mbox{Rm}}
\newcommand{\ex}{\mbox{exp}}

\newcommand{\be}{\begin{equation}}
\newcommand{\ee}{\end{equation}}
\newcommand{\bee}{\begin{equation*}}
\newcommand{\eee}{\end{equation*}}
\def\D{\Delta_f}

\def\na{\nabla}

\def\la{\langle}
\def\ra{\rangle}

\def\Pi{\displaystyle{\mathbb{II}}}

\def\a{\alpha}

\def\C{\mathbb{C}}

\begin{document}

\title[]{Curvature estimates and gap theorems for expanding Ricci solitons}

\author{Pak-Yeung Chan}

%\title[]{\LARGE C\MakeLowercase{urvature estimates and gap theorems for expanding }R\MakeLowercase{icci solitons}}

%\author{\large P\MakeLowercase{ak-}Y\MakeLowercase{eung} C\MakeLowercase{han}}
\address{School of Mathematics,
University of Minnesota, Minneapolis, MN 55455, USA} \email{chanx305@umn.edu}

\maketitle
\markboth{Pak-Yeung Chan} {Curvature estimates and gap theorems for expanding Ricci solitons}

\begin{abstract} We derive a sharp lower bound for the scalar curvature of non-flat and non-compact expanding gradient Ricci soliton provided that the scalar curvature is non-negative and the potential function is proper. Upper bound for the scalar curvature of expander with nonpositive Ricci curvature will also be given.
%We also give an upper bound for the scalar curvature of noncompact expander when the Ricci curvature is nonpositive and the potential function is proper.
Furthermore, We provide a sufficient condition for the scalar curvature of expanding soliton being nonnegative. Curvature estimates of expanding solitons in dimensions three and four will also be established. As an application, we prove a gap theorem on three dimensional gradient expander.
\end{abstract}

\section{Introduction} Let $(M^n,g)$ be an $n$ dimensional smooth connected Riemannian manifold and $X$ be a smooth vector field on $M$. The triple $(M,g,X)$ is said to be a Ricci soliton if there is a constant $\lambda$ such that the following equation is satisfied
\be\label{eq-RS-1}
\Ric+\dfrac{1}{2}L_{X}g=\lambda g,
\ee
where $\Ric$ and $L_{X}$ denote the Ricci curvature and Lie derivative with respect to $X$ respectively. A Ricci soliton is called expanding (steady, shrinking) if $\lambda<0$ $(=0, >0)$. Upon scaling the metric by a constant, we assume $\lambda \in \{-\frac{1}{2}, 0, \frac{1}{2}\}.$ The soliton is called complete if $(M,g)$ is complete as a Riemannian manifold. It is said to be gradient if $X$ can be chosen such that $X=\nabla f$ for some smooth function $f$ on $M$. In this case, $f$ is called a potential function and (\ref{eq-RS-1}) can be rewritten as
\be\label{eq-RS-2}
\Ric_{f}:=\Ric+\nabla^2 f=\lambda g.
\ee
We shall abbreviate shrinking, steady and expanding solitons as shrinker, steadier and expander respectively.

Hereinafter, $S$ denotes the scalar curvature of the Riemannian manifold. Fix a point $p_0$ in $M$, for any $x$ in $M$, $r$, $r(x)$ and $d(x, p_0)$ will be used interchangeably and refer to the distance between $x$ and $p_0$. $B_R(p)$ is the geodesic ball centered at $p$ with radius $R$. For any smooth function $\omega$, we define the weighted Laplacian w.r.t. $\omega$ by $\Delta_{\omega}:=\Delta- \na \omega\cdot \na $, where $\Delta$ is the usual Laplacian.

Ricci flow was introduced by Hamilton in his seminal work \cite{Hamilton-1982} to study closed three manifolds with positive Ricci curvature:
\be\label{eqn of RF}
\frac{\partial g(t)}{\partial t}=-2\Ric(g(t)).
\ee
\let\thefootnote\relax\footnotetext{2010 Mathematics Subject Classification. Primary 53C21; Secondary 53C20, 53C44}
Ricci soliton is of great importance since it is a self similar solution to the Ricci flow and usually arises as a rescaled limit of the flow near its singularities (see \cite{Hamilton-1995} and \cite{Cao-2010}). In particular, expanding Ricci soliton is related to the limit solution of Type III singularities of the Ricci flow (see \cite{Cao-1997}, \cite{ChenZhu-2000} and \cite{Lott-2007}). Ricci soliton is also a natural generalization of the Einstein metric. The corresponding curvature quantity $\Ric_{f}:=\Ric+\nabla^2 f$ plays a significant role in the theory of smooth metric measure spaces (see \cite{Lott-2003}, \cite{WeiWylie-2009}, \cite{MunteanuWang-2012}, \cite{MunteanuWang-2014} and \cite{MunteanuSungWang-2017}).

Chen \cite{Chen-2009} showed that any complete ancient solutions to the Ricci flow have non-negative scalar curvature. Consequently, any complete shrinking and steady gradient Ricci solitons have non-negative scalar curvature (see also \cite{Zhang-2009} for a different proof). By the strong minimum principle, the scalar curvature of gradient shrinker and steadier must be positive unless the manifolds are Ricci flat (and hence is flat in the former case). Chow-Lu-Yang \cite{ChowLuYang-2011} applied an estimate of the potential function for gradient shrinker by Cao-Zhou \cite{CaoZhou-2010} and obtained a sharp positive lower bound for the scalar curvature of noncompact and nonflat shrinker.
\begin{thm}\cite{ChowLuYang-2011}\label{scalar lower bdd in shrinker} Let $(M^n, g, f)$ be an $n$ dimensional complete noncompact and nonflat shrinking gradient Ricci soliton. Then there exists a positive constant $c$ such that
\be\label{ineq for scalar lower bdd in shrinker}
S\geq \frac{c}{r^2}
\ee
outside a compact set of $M$.
\end{thm}
 \begin{remark}As pointed out in \cite{ChowLuYang-2011}, the estimate (\ref{ineq for scalar lower bdd in shrinker}) is sharp on the shrinker constructed by Feldman-Ilmanen-Knopf \cite{FeldmanIlmanenKnopf}.
 \end{remark}
For the case of steady soliton, Chow-Lu-Yang \cite{ChowLuYang-2011} also provided a sharp lower estimate for the scalar curvature under some conditions on the potential function (see also \cite{MunteanuSungWang-2017}).
\begin{thm}\cite{ChowLuYang-2011}\label{scalar lower bdd in steadier} Let $(M^n, g, f)$ be an $n$ dimensional complete noncompact and non Ricci flat steady gradient Ricci soliton with the scaling convention $|\na f|^2+S=1$. Suppose $f\leq 0$ and $\lim_{x\to \infty} f(x)=-\infty$. Then
\be\label{ineq for scalar lower bdd in steadier}
S\geq \frac{1}{\sqrt{\frac{n}{2}}+2}e^f \text{  on } M.
\ee
\end{thm}
 \begin{remark} Using some localization and cut off function arguments, Munteanu-Sung-Wang \cite{MunteanuSungWang-2017} showed $S\geq C e^f$ without assuming $\lim_{r\to \infty} f=-\infty$. However the constant $C$ depends not only on the dimension $n$, but also the metric $g$. The decay rate in (\ref{ineq for scalar lower bdd in steadier}) is sharp on cigar soliton (see \cite{ChowLuYang-2011}, \cite{Hamilton-1988} and \cite{Chowetal-2007}).
 \end{remark}

 Unlike shrinker and steadier, expander may have negative scalar curvature. Pigola-Rimoldi-Setti \cite{PigolaRimoldiSetti-2011} and S. J. Zhang \cite{Zhang2-2011} independently improved an estimate in \cite{Zhang-2009} and showed the scalar curvature $S$ of an $n$ dimensional
 complete expanding gradient Ricci soliton satisfies
\be\label{general bdd for S in expander}
S\geq -\frac{n}{2},
\ee
with equality holds somewhere if and only if the expander is Einstein with scalar curvature $-\frac{n}{2}$ (see \cite{PigolaRimoldiSetti-2011} and \cite{Zhang2-2011}).

Roughly speaking, the inequalities (\ref{ineq for scalar lower bdd in shrinker}) and (\ref{ineq for scalar lower bdd in steadier}) in the above theorems provide some quantitative estimates to measure the gap between the solitons and the flat or Ricci-flat spaces.
It is natural to ask whether such a gap exists or not if the scalar curvature of the expander is nonnegative. We provide an affirmative answer to the above question if the potential function $f$ is proper. No upper boundedness of the scalar curvature is assumed. Here is the main result of this paper.
\begin{thm}\label{scalar lower bdd in expander}Let $(M^n,g,f)$ be an $n$ dimensional complete non-compact gradient expanding Ricci soliton with $n\geq 2$. Assume that $\lim_{r\to \infty} f=-\infty$ and $(M,g)$ is non-flat and has non-negative scalar curvature $S$. Then there exists a positive constant $C$ such that
\be\label{ineq for scalar lower bdd in expander}
S\geq Cv^{1-\frac{n}{2}}e^{-v} \text{  on  } M,
\ee
where $v:=\frac{n}{2}-f$.
\end{thm}
\begin{remark} \label{suff for proper of f in general dim exp} The estimate (\ref{ineq for scalar lower bdd in expander}) is sharp when $n=2$ on $2$ dimensional positively curved gradient expander and when $n=2m$, $m\geq 2$ on the K\"{a}hler gradient expander constructed by Feldman-Ilmanen-Knopf (see \cite{FeldmanIlmanenKnopf} and \cite{Siepmann-2013}). By (\ref{general bdd for S in expander}) and (\ref{eqn of naf}), $f\leq \frac{n}{2}$, hence $f$ is proper (preimages of compact subsets are compact) if and only if $\lim_{r\to \infty} f=-\infty$. One sufficient condition for the properness of $f$ is that for some $\varepsilon>0$,
$$\Ric\geq (\varepsilon-\frac{1}{2}) g$$
outside a compact subset of $M$.
\end{remark}
\begin{remark}Under the additional assumption that $\lim_{r\to\infty} r^2|\Ric|=0$, Deruelle \cite{Deruelle-2017} also gave a lower bound (possibly zero) of $v^{\frac{n}{2}-1}e^vS$ which depends on $\liminf_{r\to\infty} v^{\frac{n}{2}-1}e^vS$. The limit $\lim_{r\to\infty} v^{\frac{n}{2}-1}e^{v-\frac{n}{2}}S$ is also related to the notion of the scalar curvature at infinity in \cite{Deruelle-2017.0}.
\end{remark}
%\begin{remark}Under the additional assumption that $\lim_{r\to\infty} r^2|\Ric|=0$, Deruelle \cite{Deruelle-2017} also gave a lower bound (possibly zero) of $v^{\frac{n}{2}-1}e^vS$ which depends on $\liminf_{r\to\infty} v^{\frac{n}{2}-1}e^vS$. The limit $\lim_{r\to\infty} v^{\frac{n}{2}-1}e^{v-\frac{n}{2}}S$ is also related to the notion of the scalar curvature at infinity in \cite{Deruelle-2017.0}.
%\end{remark}
 For expander with nonpositive Ricci curvature and proper potential function, we prove an upper bound for the scalar curvature.
\begin{thm}\label{scalar lower bdd in expander when Ric <0}Let $(M^n,g,f)$ be an $n$ dimensional complete non-compact gradient expanding Ricci soliton with $n\geq 2$. Suppose that $\lim_{r\to \infty} f=-\infty$ and $(M,g)$ is non-flat and has non-positive Ricci curvature. Then there is a positive constant $C$ such that
\be\label{ineq for scalar lower bdd in expander when Ric <0}
S\leq -Cv^{1-\frac{n}{2}}e^{-v},
\ee
where $v:=\frac{n}{2}-f$.
\end{thm}
\begin{remark}
The upper estimate is sharp on two dimensional negatively curved gradient expanders with curvature decaying to zero at infinity.
\end{remark}

In view of Theorem \ref{scalar lower bdd in expander}, it is interesting to see when the scalar curvature of an expander is nonnegative. Sufficient condition for $S\geq 0$ on Yamabe soliton was studied by Ma-Miquel \cite{MaMiquel-2012}. The lower bound of the scalar curvature is also important since it is related to the connectedness at infinity of the expander (see \cite{MunteanuWang-2012}, \cite{ChenDeruelle-2015} and \cite{MunteanuWang-2015.4}). It is known that there doesn't exist any compact expander with nonnegative scalar curvature. Indeed, any compact expander must be Einstein with scalar curvature $-\frac{n}{2}$ (see \cite{Chowetal-2007}). Motivated by various studies of solitons with integrable curvature (see \cite{PigolaRimoldiSetti-2011}, \cite{Deruelle-2012}, \cite{CatinoMastroliaMonticelli-2016}, \cite{MunteanuSungWang-2017} and \cite{Chan-2019.2}), we prove an analog in the noncompact case. Here is the second main result of the paper.
\begin{thm}\label{sufficient condition for S geq 0}
Let $(M^n,g,f)$ be a complete noncompact gradient expanding Ricci soliton with dimension $n\geq 3$. We denote the negative part of the scalar curvature $S$ by $S_-$, i.e. $S_-:=(-S)_+:=\max{\{-S, 0\}}$. If $S_-$ is integrable, that is,
\be\label{negative of S is integrable}
\int_M S_- dv_g <\infty,
\ee
then either $S\geq 0$ everywhere or M is isometric to an Einstein manifold with scalar curvature $-\frac{n}{2}$ and finite volume, where $dv_g$ is the volume element with respect to the metric $g$.
\end{thm}
\begin{remark}
The curvature of negatively curved Cao's K\"{a}hler expander in complex dimension one decays exponentially in $r^2$ (see \cite{Cao-1997} and \cite{DengZhu-2015}). Hence Theorem \ref{sufficient condition for S geq 0} doesn't hold in real dimension two.
\end{remark}
One immediate consequence of the above theorem is that any complete noncompact gradient expander with nonnegative scalar curvature outside compact subset must have nonnegative scalar curvature everywhere. Using the asymptotic curvature estimates in \cite{Deruelle-2017}, we also have:
\begin{cor}\label{non -ve scalar bdd for Ricci flat conical expander}
Let $(M^n,g,f)$ be a complete noncompact gradient expanding Ricci soliton with dimension $n\geq 3$. Suppose in addition that
$$\lim_{x\to \infty}r^2|\Ric|=0,$$
where $r$ is the distance function from a fixed point. Then $M$ has nonnegative scalar curvature. In particular, $M$ is connected at infinity.
\end{cor}
It was proven by Pigola-Rimoldi-Setti \cite{PigolaRimoldiSetti-2011} that any gradient expander with $S\geq 0$ and $S\in L^1(M, e^{-f}dv_g)$ must be flat. Using $f\leq \frac{n}{2}$ on $M$, their result and Theorem \ref{sufficient condition for S geq 0}, we get a slightly more general classification for $n\geq 3$.
\begin{cor} Given $(M^n, g, f)$ a complete noncompact gradient expanding Ricci soliton with dimension $n\geq 3$. If the scalar curvature $S$ is in $L^1(M, e^{-f}dv_g)$, then M is isometric either to $\R^n$ or to an Einstein manifold with scalar curvature $-\frac{n}{2}$ and finite volume.
\end{cor}
Using Theorem \ref{sufficient condition for S geq 0}, we give another sufficient condition for nonnegativity of the scalar curvature.
\begin{thm}\label{non -ve scalar bdd preserved for conical data}
Let $(M^n,g,f)$ be a complete noncompact gradient expanding Ricci soliton with dimension $n\geq 3$ and bounded scalar curvature $S$. Suppose the following conditions are satisfied:
\begin{enumerate}
\item $\lim_{x\to\infty}f=-\infty$;
\item \label{scalar curvature of cone at infinity}$\lim\inf_{x\to\infty}r^2S\geq0$.
\end{enumerate}
Then the scalar curvature $S\geq 0$ on $M$. In particular, $M$ is connected at infinity.
\end{thm}
\begin{remark} \label{r2 is optimal} It can be seen from the proof of Theorem \ref{non -ve scalar bdd preserved for conical data} that $S\geq 0$ holds under a weaker assumption on $S$, namely $\limsup_{x\to\infty}v^{-1}S<1$ instead of $S$ being bounded, where $v=\frac{n}{2}-f$. Under the conditions of Corollary \ref{non -ve scalar bdd for Ricci flat conical expander} or Theorem \ref{non -ve scalar bdd preserved for conical data}, it follows from Theorem \ref{scalar lower bdd in expander} that the scalar curvature $S$ satisfies (\ref{ineq for scalar lower bdd in expander}) provided that $M$ is not flat. In view of the negatively curved Bryant's expander and Cao's K\"{a}hler expander, we know that the quadratic factor $r^2$ in Corollary \ref{non -ve scalar bdd for Ricci flat conical expander} and Theorem \ref{non -ve scalar bdd preserved for conical data} is sharp (see \cite{Cao-1997}, \cite{Bryant-2005}, \cite{Chowetal-2007} and \cite{DengZhu-2015}). We also see from the F.I.K. expander \cite{FeldmanIlmanenKnopf} that analogous result is not true if one replaces the scalar curvature $S$ by the Ricci curvature $\Ric$.
% or the curvature operator $\Rm$.
\end{remark}

We then study the curvature estimates for expanders of low dimensions, namely dimensions $3$ and $4$. Motivated by the estimates of the curvature tensor of shrinker and steadier in \cite{MunteanuWang-2015.3}, \cite{CaoCui-2014}, \cite{MunteanuWang-2017} and \cite{MunteanuWang-2016}, we prove some analogs for $3$ dimensional gradient expander.

\begin{thm}\label{curv estimate in dim 3 exp} Let $(M^3,g,f)$ be a $3$ dimensional complete non-compact gradient expanding Ricci soliton. Then the curvature tensor $\Rm$ is bounded if the scalar curvature $S$ is bounded. Moreover, the following hold:
\begin{enumerate}[label=(\alph*)]
\item \label{3 dim S control Rm exp}If $S$ is nonnegative and bounded, then
$$|\Rm|\leq c\sqrt{S} \text{  on  } M,$$
for some positive constant $c$;
\item \label{3 dim S decay implies Rm decay exp}If $S\to 0$ as $x\to\infty$, then $\lim_{x\to\infty}|\Rm|=0$.
\end{enumerate}

\end{thm}

%\begin{thm}\label{3 dim S control Rm exp} Let $(M^3,g,f)$ be a $3$ dimensional complete non-compact gradient expanding Ricci soliton with bounded non-negative scalar curvature. Then the curvature tensor $\Rm$ is bounded, moreover
%\be
%|\Rm|\leq c\sqrt{S} \text{  on  } M,
%\ee
%for some positive constant $c$.
%\end{thm}
%\begin{thm}\label{3 dim S decay implies Rm decay exp} Let $(M^3, g, f)$ be a $3$ dimensional complete noncompact gradient expanding Ricci soliton. Suppose $\lim_{x\to\infty} S(x)=0$. Then
%$$\lim_{x\to\infty}|Rm|(x)=0.$$
%\end{thm}
As an application, we show that $f$ is proper if the scalar curvature $S$ decays at infinity in dimension three.
\begin{cor}\label{proper f in 3 dim}Let $(M^3,g,f)$ be a $3$ dimensional complete non-compact gradient expanding Ricci soliton. Suppose that $\displaystyle\lim_{x\to \infty} S=0$. Then $\displaystyle\lim_{x\to \infty} f=-\infty$, furthermore
\be\label{f is proper eqn in 3 dim exp}
\lim_{x\to \infty} \frac{4f(x)}{r^2(x)}=-1.
\ee
\end{cor}
\begin{remark}In view of Theorem \ref{n dim S control Ric if Rm bdd exp} and Proposition \ref{v equiv r2 under Ric lower bdd decay to 0}, we see that (\ref{f is proper eqn in 3 dim exp}) is also true in higher dimensions if in addition $|\Rm|$ is bounded on $M$.
\end{remark}
Munteanu-Wang \cite{MunteanuWang-2015.3} showed that any $4$ dimensional complete gradient shrinker must have bounded curvature if the scalar curvature is bounded (see also \cite{CaoCui-2014} for the estimates in steady soliton). We prove that it is also true for gradient expanding soliton if in addition the potential function $f\to -\infty$ as $r\to \infty$.
\begin{thm}\label{bdd S and f proper implies bdd Rm}Let $(M^4,g,f)$ be a $4$ dimensional complete non-compact gradient expanding Ricci soliton with bounded scalar curvature. Suppose that $\lim_{r\to \infty} f=-\infty.$
Then the curvature tensor $\Rm$ is bounded.
\end{thm}
%The paper is organized as follow. We include all the preliminaries in section $2$. Theorem \ref{scalar lower bdd in expander} and \ref{scalar lower bdd in expander when Ric <0} will then be proved in section $3$ and $4$ respectively. In appendix, we shall show lemmas \ref{compute lemma 1} and \ref{compute lemma 2}.
It was proved by Deng-Zhu \cite{DengZhu-2015} and Deruelle \cite{Deruelle-2017} that any gradient expander with non-negative Ricci curvature must have bounded scalar curvature (their arguments work well if $\Ric \geq 0$ outside some compact subset of $M$). Together with Remark \ref{suff for proper of f in general dim exp}, we have the following corollary.
\begin{cor}
Let $(M^4,g,f)$ be a $4$ dimensional complete noncompact gradient expanding Ricci soliton with $\Ric\geq 0$ outside some compact subset of $M$. Then it has bounded curvature tensor.
\end{cor}
Using Shi's estimate \cite{ChowLuNi-2006}, Theorem \ref{n dim S control Ric if Rm bdd exp} and Lemma \ref{Rm to Rc}, we establish the equivalence of the properness of $f$ and the curvature decay at infinity.
\begin{cor}
Let $(M^4,g,f)$ be a $4$ dimensional complete noncompact gradient expanding Ricci soliton with $\lim_{x\to\infty}S(x)=0$. Then $\lim_{x\to\infty} f(x)=-\infty$ if and only if $\lim_{x\to\infty}|\Rm|(x)=0$.
\end{cor}
%As an application of the previous curvature estimates, we study the gap theorem of expanding soliton. Gap phenomena of manifolds have long been an interesting problem in Geometry. A special case of the problem is concerned with the flatness of noncompact manifold, more precisely, for a non-compact manifold with curvature decaying sufficiently fast at infinity, is it necessarily flat? Using the structure of expanding Ricci solitons, Chen and Deruelle \cite{ChenDeruelle-2015} showed that any gradient expander with $\Ric\geq 0$ and $\lim_{r\to \infty} r^2|\Rm|=0$ is flat. Deng and Zhu \cite{DengZhu-2015} proved that a K\"{a}hler gradient expander with complex dimensions $n\geq 2$, $\Ric\geq 0$ and $\lim_{r\to \infty} r^2S=0$ is isometric to $\C^n$ (see \cite{MokSiuYau-1981}, \cite{GreeneWu-1982} and \cite{BandoKasueNakajima-1989} for more results on the gap theorems of general Riemannian manifolds). In either cases, Ricci curvature is assumed to be non-negative. Using the positive mass theorem, asymptotic curvature estimates in \cite{Deruelle-2017}, Theorems \ref{non -ve scalar bdd preserved for conical data} and \ref{curv estimate in dim 3 exp}\ref{3 dim S decay implies Rm decay exp}, we are able to prove a gap theorem for $3$ dimensional gradient expander by only imposing conditions on the scalar curvature.

As an application of the previous curvature estimates, we study the gap theorem of expanding soliton. Gap phenomenon of manifold has long been an interesting problem in geometry. A special case of the problem is concerned with the flatness of noncompact manifold, more precisely, for a non-compact manifold with curvature decaying sufficiently fast at infinity, is it necessarily flat? The tangent cone at infinity of gradient expander with $\lim_{x\to \infty}r^2|\Rm|=0$ was proven to be flat by Chen \cite{Chen-2012} and Chen-Deruelle \cite{ChenDeruelle-2015}. The expander itself is also flat if in addition $\Ric\geq 0$ (see \cite{BandoKasueNakajima-1989} and \cite{ChenDeruelle-2015}). Deng and Zhu \cite{DengZhu-2015} proved that a K\"{a}hler gradient expander with complex dimension $n\geq 2$, $\Ric\geq 0$ and $\lim_{r\to \infty} r^2S=0$ is isometric to $\C^n$ (see \cite{MokSiuYau-1981}, \cite{GreeneWu-1982} and \cite{BandoKasueNakajima-1989} for more results on the gap theorems of general Riemannian manifolds). In either cases, Ricci curvature is assumed to be non-negative. Using the positive mass theorem, asymptotic curvature estimates in \cite{Deruelle-2017}, Theorems \ref{sufficient condition for S geq 0} and \ref{curv estimate in dim 3 exp}\ref{3 dim S decay implies Rm decay exp}, we are able to prove a gap theorem for $3$ dimensional gradient expander by only imposing condition on the scalar curvature.

%\begin{thm}\label{gap thm for 3 dim exp} Let $(M^3,g,f)$ be a $3$ dimensional complete non-compact gradient expanding Ricci soliton with non-negative scalar curvature. Suppose that
%$$\displaystyle\lim_{x\to \infty} r^2(x)S(x)=0,$$
%then $M$ is isometric to $\R^{3}$.
%\end{thm}
\begin{thm}\label{gap thm for 3 dim exp} Let $(M^3,g,f)$ be a $3$ dimensional complete non-compact gradient expanding Ricci soliton. Suppose the scalar curvature $S$ satisfies
$$\displaystyle\lim_{x\to \infty} r^2(x)S(x)=0.$$
Then $M$ is isometric to $\R^{3}$.
\end{thm}
\begin{remark}
Analogous result is not true in higher even dimensions. The non-flat K\"{a}hler gradient expander constructed by Feldman-Ilmanen-Knopf \cite{FeldmanIlmanenKnopf} has nonnegative scalar curvature which decays exponentially in $r$.
\end{remark}
\begin{remark} Positive mass theorem was also used by Ma \cite{Ma-2011} to prove the flatness of $3$ dimensional Ricci pinched gradient expander. Using the asymptotic curvature estimates \cite{Deruelle-2017} and volume comparison theorem, Deruelle showed that any $4$ dimensional Ricci pinched gradient expander is flat (see \cite{Ma-2011} and \cite{Deruelle-2017} for more results in higher dimensions under additional assumptions). Again $\Ric\geq 0$ was assumed in both cases.
\end{remark}
Recently the positive mass theorem for asymptotically flat manifold with dimension $n\geq 9$ has been proven by Lohkamp \cite{Lohkamp-2016.0}, \cite{Lohkamp-2016}, and later also by Schoen-Yau \cite{SchoenYau-2017}. Hence similar argument for Theorem \ref{gap thm for 3 dim exp} gives the following gap theorem in higher odd dimensions without sign condition on the Ricci curvature.

%The tangent cone at infinity of expander with $\lim_{x\to \infty}r^2|\Rm|=0$ was proven to be flat in \cite{Chen-2012} and \cite{ChenDeruelle-2015}. In odd dimensions, we show that the expander itself has to be flat without imposing sign condition on the Ricci curvature.
%\begin{thm}Let $(M^n, g, f)$ be a complete noncompact gradient expanding Ricci soliton with dimension $n\geq 4$. If $M$ is asymptotically flat and $\lim_{x\to\infty} r^2|\Rm|=0$, then $M$ is isometric to $\R^n$.
%\end{thm}
%\begin{remark} By a result in \cite{Deruelle-2017}, $\lim_{x\to\infty} r^2|\Rm|=0$ implies that $M$ is asymptotically locally Euclidean. Hence asymptotically flatness condition is superfluous if dimension $n$ is odd.
%\end{remark}

\begin{thm}\label{gap thm for odd dim exp}Let $(M^{2m+1}, g, f)$ be a complete noncompact gradient expanding Ricci soliton with odd dimension $2m+1$ and $m\geq 2$. If $\lim_{x\to\infty} r^2|\Rm|=0$, then $M$ is isometric to $\R^{2m+1}$.
\end{thm}
\begin{remark} The theorem is not true in even dimensions. The non-flat Feldman-Ilmanen-Knopf K\"{a}hler expander \cite{FeldmanIlmanenKnopf} has flat asymptotic cone and satisfies $\lim_{x\to\infty} r^2|\Rm|=0$. Nonetheless, if in addition either one of the following conditions holds:
\begin{enumerate}
\item $M$ is smoothly asymptotic to the cone $(C(\mathbb{S}^{n-1}(1)), dt^2+t^2g_{\mathbb{S}^{n-1}(1)})$ in sense of \cite{Deruelle-2017};
\item $\liminf_{R\to\infty}\frac{Vol_g(B_R(p_0))}{R^n}>\frac{\omega_{n-1}}{2n}$, where  $\omega_{n-1}$ is the volume of $\mathbb{S}^{n-1}(1)$ in $\R^{n}$ with respect to the Euclidean metric;
\item $M$ is simply connected at infinity;
\item $\Ric\geq \big(\delta-\frac{1}{2}\big)g$ on $M$ for some $\delta>0$,
\end{enumerate}
then $M^{n}$ is isometric to $\R^{n}$, where $n=2m$ and $m\geq 2$. As mentioned in Remark \ref{r2 is optimal}, the quadratic factor $r^2$ in Theorems \ref{gap thm for 3 dim exp} and \ref{gap thm for odd dim exp} is optimal.
\end{remark}

The paper is organized as follows. We include all the preliminaries and computations for the proof of Theorems \ref{scalar lower bdd in expander} and \ref{scalar lower bdd in expander when Ric <0} in Section $2$. Theorems \ref{scalar lower bdd in expander} and \ref{scalar lower bdd in expander when Ric <0} will then be proved in Sections $3$ and $4$ respectively. In Section $5$, We give a proof for Theorem \ref{sufficient condition for S geq 0} using integration by part. In Section $6$, we shall show Theorems \ref{curv estimate in dim 3 exp} and \ref{gap thm for 3 dim exp}. We then justify Theorem \ref{non -ve scalar bdd preserved for conical data} in Section $7$. The proof of Theorem \ref{bdd S and f proper implies bdd Rm} will be given in the last section.

{\sl Acknowledgement}: The author would like to express his sincere appreciation to his advisor, Professor Jiaping Wang for his constant encouragement and stimulating discussion. He would also like to thank Professor Ovidiu Munteanu for his comments and interest in this work.% The author is greatly indebted to Professor Liang 

\section{preliminaries}
Let $(M^n,g,f)$ be an $n$ dimensional gradient expanding Ricci soliton, i.e.
$$\Ric+\na^2 f=-\frac{1}{2}g.$$
It generates a self similar solution to the Ricci flow. Indeed, let $\psi_t$ be the flow of the vector field $\frac{\na f}{1+t}$ with $\psi_0$ being the identity map. We define $g(t):=(1+t)\psi_t^*g$, then $g(t)$ is a solution to the Ricci flow for $t$ $\in (-1,\infty)$ with $g(0)=g$. The following equations for gradient expanders are known (see \cite{Chowetal-2007}, \cite{PetersenWylie-2010} and \cite{Deruelle-2017})
\be\label{trace RS eqn}
S+\Delta f= -\frac{n}{2},
\ee
\be \label{eqn of naf}
S+|\na f|^2=-f,
\ee
\be\label{eqn of f}
\Delta_f f=f-\frac{n}{2},
\ee
%\be\label{na S equal 2 Ric na f}
%\na S=2\Ric(\na f) \text{ and }
%\ee
\be\label{eqn of S}
\D S=-S-2|\Ric|^2,
\ee

\be\label{divRm}
R_{kj,i}-R_{ki,j}=R_{ijkl}f_l,
\ee

\be\label{eqn of Rc in exp}
\D R_{ij}=-R_{ij}-2R_{iklj}R_{kl}
\ee
and
\be\label{eqn of Rm in exp}
\D \Rm=-\Rm + \Rm\ast\Rm,
\ee
where $S$ is the scalar curvature and $\D:=\Delta-\na _{\na f}$. We define a function $v$ in the following way:
\be\label{defn of v}
v:= \frac{n}{2}-f.
\ee
From (\ref{eqn of naf}) and (\ref{eqn of f}), we have
\be\label{eqn of v}
\D v=v \text{    and   }
\ee
\be \label{eqn of nav}
S+|\na v|^2=v-\frac{n}{2}.
\ee
By (\ref{general bdd for S in expander}) and (\ref{eqn of nav}), we have $|\na v|^2\leq v$ and $|\na \sqrt{v+\delta}|\leq \frac{1}{2}$ for any $\delta>0$. Integrating the inequality along minimizing geodesics and letting $\delta \to 0$, we get
\be\label{naf nav sqrtv control linearly in exp}
|\na f|=|\na v|\leq \sqrt{v}\leq \frac{1}{2}r+\sqrt{v(p_0)},
\ee
where $p_0$ is a fixed point in $M$.
%\begin{lma}Let $(M^n, g, f)$ be a complete noncompact and non-Einstein expanding gradient Ricci soliton. Then $v>0$ on $M$.
%\end{lma}
%\begin{proof}\label{v proper implies v>0}
%From (\ref{general bdd for S in expander}) and (\ref{eqn of nav}), we see that $v\geq |\na v|^2\geq 0$. If $v$ vanishes somewhere, then by strong minimum principle \cite{GilbargTrudinger-2001} and (\ref{eqn of v}), $v\equiv 0$. Hence $\Ric=-\frac{g}{2}$, which is absurd.
%\end{proof}
%If $\lim_{r\to \infty}f=-\infty$ and $g$ is not flat, then by a result of \cite{PigolaRimoldiSetti-2011}, $g$ is non-Einstein.
Suppose that $S$ is now non-negative. Then we have by (\ref{eqn of naf}) that $f\leq 0$. Using strong minimum principle \cite{GilbargTrudinger-2001} and (\ref{eqn of S}), we conclude that $S>0$ unless $S\equiv 0$ on $M$ (i.e. flat, see \cite{PigolaRimoldiSetti-2011}). Moreover $v$ satisfies
\be\label{lower bdd of v}
v\geq \frac{n}{2}.
\ee
In general the scalar curvature may be negative, the lower bound of $v$ (\ref{lower bdd of v}) is no longer available. However, it is known that $v$ is positive under some mild conditions (see \cite{Deruelle-2017}).
%\begin{lma}Let $(M^n, g, f)$ be a complete noncompact and non-Einstein expanding gradient Ricci soliton. Then $v>0$ on $M$.
%\end{lma}
\begin{lma}\label{v proper implies v>0} Let $(M^n, g, f)$ be a complete noncompact expanding gradient Ricci soliton. If $\Ric \neq -\frac{g}{2}$ somewhere on $M$, then $v>0$ on $M$.
\end{lma}

\begin{proof}
From (\ref{general bdd for S in expander}) and (\ref{eqn of nav}), we see that $v\geq |\na v|^2\geq 0$. If $v$ vanishes somewhere, then by the strong minimum principle \cite{GilbargTrudinger-2001} and (\ref{eqn of v}), $v\equiv 0$. Hence $\Ric=-\frac{g}{2}$ on $M$, which is absurd.
\end{proof}
If $\lim_{r\to \infty}f=-\infty$, then $\Ric \neq -\frac{g}{2}$ somewhere and hence $v>0$. The following lemmas are immediate consequences of the computations by Deruelle \cite{Deruelle-2017}. We include the calculations for the sake of completeness.

\begin{lma}\label{compute lemma 1}
Let $(M^n, g, f)$ be an n dimensional complete gradient expanding Ricci soliton and $p\in M$. If $v(p)>0$, then at $p$
%Under the above notations,
\be\label{eqn for vS}
\Delta_{f+2\ln v}(vS)=-2v|\Ric|^2-2|\na \ln v|^2 vS.
\ee
\end{lma}
\begin{proof}
Using (\ref{eqn of S}) and (\ref{eqn of v}),
\begin{eqnarray*}
\D (vS)&=& v\D S+ S\D v +2\langle \na v,\na S\rangle\\
&=& v(-S-2|\Ric|^2)+Sv+2\langle \na v,\na (vS v^{-1})\rangle\\
&=&-2v|\Ric|^2+2\langle \na \ln v, \na (vS)\rangle-2|\na \ln v|^2 vS.
\end{eqnarray*}
This completes the proof of the lemma.
%Lemma \ref{compute lemma 1} now follows.
\end{proof}
\begin{lma}\label{compute lemma 2}Under the same assumption in Lemma \ref{compute lemma 1}, then at $p$
\begin{eqnarray}\notag
\Delta_{f+2\ln v}(e^{\frac{1}{\sqrt{v}}}v^{2-\frac{n}{2}}e^{-v}) &=&e^{\frac{1}{\sqrt{v}}}v^{2-\frac{n}{2}}e^{-v}\Big\{\frac{1}{2\sqrt{v}}-(S+\frac{n}{2})\frac{1}{v^{\frac{3}{2}}}+\frac{1}{v^{\frac{3}{2}}}(\frac{n}{2}-2)\\
\notag
&&-(S+\frac{n}{2})\frac{1}{v^{\frac{5}{2}}}(\frac{n}{2}-2)-S + (\frac{n}{2}-2)(\frac{n}{2}+1)\frac{1}{v}\\
\label{long computation}
&&-(S+\frac{n}{2})(\frac{n}{2}-2)(\frac{n}{2}+1)\frac{1}{v^2}-(2S+n)(\frac{n}{2}-1)\frac{1}{v}\\
\notag
&&+\frac{1}{v^{\frac{3}{2}}}\big[\frac{7}{4}+\frac{1}{4\sqrt{v}}\big]-\frac{1}{v^{\frac{5}{2}}}(S+\frac{n}{2})\big[\frac{7}{4}+\frac{1}{4\sqrt{v}}\big]\Big\}.
\end{eqnarray}
Consequently if $M$ is noncompact and $v\to \infty$ as $r\to \infty$, then near infinity
\be\label{short form of long computation}
\Delta_{f+2\ln v}(e^{\frac{1}{\sqrt{v}}}v^{2-\frac{n}{2}}e^{-v})=e^{\frac{1}{\sqrt{v}}}v^{2-\frac{n}{2}}e^{-v}\Big\{
\big(\frac{1}{2}+o(1)\big)\frac{1}{\sqrt{v}}+S\big(-1+o(1)\big)\Big\}.
\ee
\end{lma}
\begin{proof}
\begin{eqnarray*}
\D v^{2-\frac{n}{2}}&=&(2-\frac{n}{2})v^{1-\frac{n}{2}}\D v + (\frac{n}{2}-2)(\frac{n}{2}-1)v^{-\frac{n}{2}}|\na v|^2\\
&=&(2-\frac{n}{2})v^{2-\frac{n}{2}}+ (\frac{n}{2}-2)(\frac{n}{2}-1)v^{2-\frac{n}{2}}|\na \ln v|^2.
\end{eqnarray*}
By (\ref{eqn of nav}),
\begin{eqnarray*}
\D e^{-v}&=& -e^{-v}\D v + e^{-v}|\na v|^2\\
&=& e^{-v}(|\na v|^2-v)\\
&=& -(S+\frac{n}{2})e^{-v}.
\end{eqnarray*}
\begin{eqnarray*}
2\langle \na (v^{2-\frac{n}{2}}), \na e^{-v}\rangle &=& 2(2-\frac{n}{2})v^{1-\frac{n}{2}}|\na v|^2(-e^{-v})\\
&=& 2(\frac{n}{2}-2)v^{2-\frac{n}{2}}e^{-v}-(2S+n)(\frac{n}{2}-2)v^{1-\frac{n}{2}}e^{-v}.
\end{eqnarray*}
\begin{eqnarray*}
\D (v^{2-\frac{n}{2}}e^{-v})&=& e^{-v}\D v^{2-\frac{n}{2}}+ v^{2-\frac{n}{2}}\D e^{-v}+ 2\langle \na (v^{2-\frac{n}{2}}), \na e^{-v}\rangle\\
&=& (2-\frac{n}{2})v^{2-\frac{n}{2}}e^{-v}+ (\frac{n}{2}-2)(\frac{n}{2}-1)v^{2-\frac{n}{2}}e^{-v}|\na \ln v|^2\\
& &-(S+\frac{n}{2})v^{2-\frac{n}{2}}e^{-v}+ 2(\frac{n}{2}-2)v^{2-\frac{n}{2}}e^{-v}\\
& &-(2S+n)(\frac{n}{2}-2)v^{1-\frac{n}{2}}e^{-v}\\
&=& v^{2-\frac{n}{2}}e^{-v}\Big[-2-S + (\frac{n}{2}-2)(\frac{n}{2}-1)|\na \ln v|^2\\
& &\qquad\qquad-(2S+n)(\frac{n}{2}-2)v^{-1}\Big].
\end{eqnarray*}
On the other hand, by (\ref{eqn of nav})
\begin{eqnarray*}
-2\langle \na \ln v, \na (v^{2-\frac{n}{2}}e^{-v})\rangle &=& 2(\frac{n}{2}-2)v^{1-\frac{n}{2}}e^{-v}\langle \frac{\na v}{v}, \na v\rangle\\
& &+ 2v^{2-\frac{n}{2}}e^{-v}\langle \frac{\na v}{v}, \na v\rangle\\
&=& 2(\frac{n}{2}-2)v^{2-\frac{n}{2}}e^{-v}|\na \ln v|^2\\
&&+ 2v^{2-\frac{n}{2}}e^{-v}\frac{(v-S-\frac{n}{2})}{v}\\
&=& v^{2-\frac{n}{2}}e^{-v}\Big[2-(2S+n)v^{-1}\\
&&\qquad\qquad+2(\frac{n}{2}-2)|\na \ln v|^2\Big]
\end{eqnarray*}
Hence
\begin{eqnarray}\notag
\Delta_{f+2\ln v} (v^{2-\frac{n}{2}}e^{-v})&=& v^{2-\frac{n}{2}}e^{-v}\Big[ -S + (\frac{n}{2}-2)(\frac{n}{2}+1)|\na \ln v|^2\\
\label{half eq for e-v}&&\qquad\qquad-(2S+n)(\frac{n}{2}-1)v^{-1}\Big].
\end{eqnarray}
\begin{eqnarray*}
\D e^{\frac{1}{\sqrt{v}}}&=&-\frac{1}{2}v^{-\frac{3}{2}}e^{\frac{1}{\sqrt{v}}}\D v+ e^{\frac{1}{\sqrt{v}}}\big[\frac{3}{4}v^{-\frac{5}{2}}+\frac{1}{4}v^{-3}\big]|\na v|^2\\
&=&-\frac{1}{2}v^{-\frac{1}{2}}e^{\frac{1}{\sqrt{v}}}+ e^{\frac{1}{\sqrt{v}}}\big[\frac{3}{4}v^{-\frac{5}{2}}+\frac{1}{4}v^{-3}\big](v-S-\frac{n}{2})\\
&=&-\frac{1}{2}v^{-\frac{1}{2}}e^{\frac{1}{\sqrt{v}}}+ v^{-\frac{3}{2}}e^{\frac{1}{\sqrt{v}}}\big[\frac{3}{4}+\frac{1}{4}v^{-\frac{1}{2}}\big]\\
&&-v^{-\frac{5}{2}}e^{\frac{1}{\sqrt{v}}}(S+\frac{n}{2})\big[\frac{3}{4}+\frac{1}{4}v^{-\frac{1}{2}}\big].
\end{eqnarray*}
\begin{eqnarray*}
\Delta_{f+2\ln v}e^{\frac{1}{\sqrt{v}}}&=& \D e^{\frac{1}{\sqrt{v}}} -2\langle \frac{\na v}{v}, \na e^{\frac{1}{\sqrt{v}}}\rangle\\
&=&\D e^{\frac{1}{\sqrt{v}}}+ \frac{(v-S-\frac{n}{2})}{v^{\frac{5}{2}}}e^{\frac{1}{\sqrt{v}}}\\
&=&-\frac{1}{2}v^{-\frac{1}{2}}e^{\frac{1}{\sqrt{v}}}+ v^{-\frac{3}{2}}e^{\frac{1}{\sqrt{v}}}\big[\frac{7}{4}+\frac{1}{4}v^{-\frac{1}{2}}\big]\\
&&-v^{-\frac{5}{2}}e^{\frac{1}{\sqrt{v}}}(S+\frac{n}{2})\big[\frac{7}{4}+\frac{1}{4}v^{-\frac{1}{2}}\big].
\end{eqnarray*}
\begin{eqnarray*}
2\langle \na e^{\frac{1}{\sqrt{v}}}, \na (v^{2-\frac{n}{2}}e^{-v})\rangle &=&-v^{-\frac{3}{2}}e^{\frac{1}{\sqrt{v}}}\langle \na v, \na v \rangle(2-\frac{n}{2})v^{1-\frac{n}{2}}e^{-v}\\
& &-v^{-\frac{3}{2}}e^{\frac{1}{\sqrt{v}}}\langle \na v, \na v \rangle v^{2-\frac{n}{2}}(-e^{-v})\\
&=&v^{-\frac{3}{2}}e^{\frac{1}{\sqrt{v}}}(v-S-\frac{n}{2})(\frac{n}{2}-2)v^{1-\frac{n}{2}}e^{-v}\\
& &+v^{-\frac{3}{2}}e^{\frac{1}{\sqrt{v}}}(v-S-\frac{n}{2}) v^{2-\frac{n}{2}}e^{-v}\\
&=&v^{-\frac{1}{2}}e^{\frac{1}{\sqrt{v}}}v^{2-\frac{n}{2}}e^{-v}-(S+\frac{n}{2})v^{-\frac{3}{2}}e^{\frac{1}{\sqrt{v}}}v^{2-\frac{n}{2}}e^{-v}\\
& &+v^{-\frac{3}{2}}e^{\frac{1}{\sqrt{v}}}(\frac{n}{2}-2)v^{2-\frac{n}{2}}e^{-v}\\
&&-(S+\frac{n}{2})v^{-\frac{5}{2}}e^{\frac{1}{\sqrt{v}}}(\frac{n}{2}-2)v^{2-\frac{n}{2}}e^{-v}.
\end{eqnarray*}
Using (\ref{half eq for e-v}), we have
\begin{eqnarray*}
\Delta_{f+2\ln v}(e^{\frac{1}{\sqrt{v}}}v^{2-\frac{n}{2}}e^{-v})&=&e^{\frac{1}{\sqrt{v}}}\Delta_{f+2\ln v}(v^{2-\frac{n}{2}}e^{-v})+v^{2-\frac{n}{2}}e^{-v}\Delta_{f+2\ln v}(e^{\frac{1}{\sqrt{v}}})\\
& &+ 2\langle \na e^{\frac{1}{\sqrt{v}}}, \na (v^{2-\frac{n}{2}}e^{-v})\rangle\\
&=&e^{\frac{1}{\sqrt{v}}}v^{2-\frac{n}{2}}e^{-v}\Big\{\frac{1}{2\sqrt{v}}-(S+\frac{n}{2})v^{-\frac{3}{2}}+v^{-\frac{3}{2}}(\frac{n}{2}-2)\\
&&-(S+\frac{n}{2})v^{-\frac{5}{2}}(\frac{n}{2}-2)-S + (\frac{n}{2}-2)(\frac{n}{2}+1)|\na \ln v|^2\\
&&-(2S+n)(\frac{n}{2}-1)v^{-1}+v^{-\frac{3}{2}}\big[\frac{7}{4}+\frac{1}{4}v^{-\frac{1}{2}}\big]\\
&&-v^{-\frac{5}{2}}(S+\frac{n}{2})\big[\frac{7}{4}+\frac{1}{4}v^{-\frac{1}{2}}\big]\Big\}\\
&=&e^{\frac{1}{\sqrt{v}}}v^{2-\frac{n}{2}}e^{-v}\Big\{\frac{1}{2\sqrt{v}}-(S+\frac{n}{2})v^{-\frac{3}{2}}+v^{-\frac{3}{2}}(\frac{n}{2}-2)\\
&&-(S+\frac{n}{2})v^{-\frac{5}{2}}(\frac{n}{2}-2)-S + (\frac{n}{2}-2)(\frac{n}{2}+1)v^{-1}\\
&&-(S+\frac{n}{2})(\frac{n}{2}-2)(\frac{n}{2}+1)v^{-2}-(2S+n)(\frac{n}{2}-1)v^{-1}\\
&&+v^{-\frac{3}{2}}\big[\frac{7}{4}+\frac{1}{4}v^{-\frac{1}{2}}\big]-v^{-\frac{5}{2}}(S+\frac{n}{2})\big[\frac{7}{4}+\frac{1}{4}v^{-\frac{1}{2}}\big]\Big\}.
\end{eqnarray*}
We showed (\ref{long computation}). We can then separate the terms with $S$ from those without $S$ in R.H.S of (\ref{long computation}) to get (\ref{short form of long computation}).
\end{proof}

\section{Proof of Theorem \ref{scalar lower bdd in expander}}
With all the computations in the previous section, we are going to prove Theorem \ref{scalar lower bdd in expander}:
\begin{thm*}Let $(M^n,g,f)$ be an $n$ dimensional complete non-compact gradient expanding Ricci soliton with $n\geq 2$. Suppose that $\lim_{r\to \infty} f=-\infty$ and $(M,g)$ is not flat and has non-negative scalar curvature $S$. Then there exists a positive constant $C$ such that
$$S\geq Cv^{1-\frac{n}{2}}e^{-v} \text{  on  } M,$$
where $v:=\frac{n}{2}-f$.
\end{thm*}
\begin{proof}%[\textbf{Proof of theorem \ref{scalar lower bdd in expander}:}]
Since the scalar curvature $S\geq 0$ and $v\geq \frac{n}{2}$, we have by (\ref{eqn for vS})
\be\label{ineq of vS}
\Delta_{f+2\ln v}(vS)\leq 0.
\ee
%Moreover by $\lim_{x\to\infty} v(x)=\infty$, we can separate the terms with $S$ from those without $S$ in R.H.S of (\ref{long computation}) to see that outside a compact subset of $M$,
Moreover by $\lim_{x\to\infty} v(x)=\infty$ and (\ref{short form of long computation}), we see that outside a compact subset of $M$,
\begin{eqnarray*}
\Delta_{f+2\ln v}(e^{\frac{1}{\sqrt{v}}}v^{2-\frac{n}{2}}e^{-v})&=&e^{\frac{1}{\sqrt{v}}}v^{2-\frac{n}{2}}e^{-v}\Big\{
\big(\frac{1}{2}+o(1)\big)\frac{1}{\sqrt{v}}+S\big(-1+o(1)\big)\Big\}.
\end{eqnarray*}
From the above equation,
\be\label{ineq of e-v}
\Delta_{f+2\ln v}(e^{\frac{1}{\sqrt{v}}}v^{2-\frac{n}{2}}e^{-v})\geq e^{\frac{1}{\sqrt{v}}}v^{2-\frac{n}{2}}e^{-v}(\frac{1}{4\sqrt{v}}-2S).
\ee
Since $v=\frac{n}{2}-f \to \infty$ as $r\to\infty$, by taking a larger compact set if necessary, we may assume that
\be\label{contradict ineq}
\frac{1}{8}>e^{\frac{1}{\sqrt{v}}}v^{\frac{3}{2}-\frac{n}{2}}e^{-v}
\ee
near infinity. Let $R_0$ be a large positive number such that (\ref{ineq of vS}), (\ref{ineq of e-v}) and (\ref{contradict ineq}) hold on $M\setminus B_{R_0}(p_0)$. Hence by $S>0$ and (\ref{lower bdd of v}), there exists a constant $b\in (0,1)$ such that
\be\label{Q at bdry R}
vS>be^{\frac{1}{\sqrt{v}}}v^{2-\frac{n}{2}}e^{-v} \text{  on  } \partial B_{R_0}(p_0).
\ee
Let $Q:=vS-be^{\frac{1}{\sqrt{v}}}v^{2-\frac{n}{2}}e^{-v}$. It is not difficult to see that $$\liminf_{r\to \infty} Q\geq 0.$$
Fix any $y$ in $M\setminus \overline{B_{R_0}(p_0)}$ and any $\varepsilon>0$, there exists a large positive $T>R_0$ such that $y$ $\in$ $B_T(p_0)$ and
\be \label{Q at bdry T}
Q\geq -\varepsilon \text{  on  } \partial B_T(p_0).
\ee
Let $\Omega:= B_T(p_0)\setminus\overline{B_{R_0}(p_0)}$ and $z$ $\in$ $\overline{\Omega}$ such that $Q$ attains its minimum over $\overline{\Omega}$ at $z$, i.e.
$$Q(z)=\min_{\overline{\Omega}}Q.$$
If $z$ $\in$ $\partial \Omega$, then by (\ref{Q at bdry R}) and (\ref{Q at bdry T}), we have
$$Q(y)\geq Q(z)\geq -\varepsilon.$$
If $z$ $\in$ $\Omega$, we have by (\ref{ineq of vS}) and (\ref{ineq of e-v}) that at $z$
\begin{eqnarray*}
0&\leq& \Delta_{f+2\ln v} Q\\
&=&\Delta_{f+2\ln v} (vS) -b\Delta_{f+2\ln v} (e^{\frac{1}{\sqrt{v}}}v^{2-\frac{n}{2}}e^{-v})\\
&\leq& -b\Delta_{f+2\ln v} (e^{\frac{1}{\sqrt{v}}}v^{2-\frac{n}{2}}e^{-v})\\
&\leq& 2b e^{\frac{1}{\sqrt{v}}}v^{2-\frac{n}{2}}e^{-v}(S-\frac{1}{8\sqrt{v}}).
\end{eqnarray*}
Together with (\ref{contradict ineq}), we know that at $z$
\begin{eqnarray*}
vS&\geq& \frac{\sqrt{v}}{8}\\
&\geq& e^{\frac{1}{\sqrt{v}}}v^{2-\frac{n}{2}}e^{-v}\\
&\geq& b e^{\frac{1}{\sqrt{v}}}v^{2-\frac{n}{2}}e^{-v},
\end{eqnarray*}
i.e. $0\leq Q(z)\leq Q(y)$. Hence in any cases, $Q(y)\geq -\varepsilon$. Result then follows by letting $\varepsilon \to 0$ and choosing a smaller $b$ to make $Q\geq 0$ on the entire $M$.
\end{proof}
\section{Proof of Theorem \ref{scalar lower bdd in expander when Ric <0}}

%\begin{lma}Let $(M^n, g, f)$ be a complete noncompact expanding gradient Ricci soliton with $\lim_{r\to\infty}f=-\infty$. Then $\min_M v>0$ on $M$.
%\end{lma}
%\begin{proof}\label{v proper implies v>0} Since $v=\frac{n}{2}-f\to\infty$ as $r\to\infty$, $v$ attains global minimum at some point in $M$, say $z$. By %strong minimum principle \cite{GilbargTrudinger-2001} and (\ref{eqn of v}), $v(z)>0$.
%\end{proof}
%Since the scalar curvature is non-positive, the lower bound of $v$ (\ref{lower bdd of v}) is no longer available. However, we still can show that $v$ is positive under some mild conditions.
%\begin{lma}Let $(M^n, g, f)$ be a complete noncompact and non-Einstein expanding gradient Ricci soliton. Then $v>0$ on $M$.
%\end{lma}
%\begin{proof}\label{v proper implies v>0}
%From (\ref{general bdd for S in expander}) and (\ref{eqn of nav}), we see that $v\geq |\na v|^2\geq 0$. If $v$ vanishes somewhere, then by strong minimum principle \cite{GilbargTrudinger-2001} and (\ref{eqn of v}), $v\equiv 0$. Hence $\Ric=-\frac{g}{2}$, which is absurd.
%\end{proof}
%If $\lim_{r\to \infty}f=-\infty$ and $g$ is not flat, then by a result of \cite{PigolaRimoldiSetti-2011}, $g$ is non-Einstein.
To prepare for the maximum principle argument, we first show that the scalar curvature is negative in a gradient expander with nonpositive Ricci curvature.
\begin{lma}\label{Ric<0 implies S<0}Let $(M^n, g, f)$ be a complete noncompact and nonflat expanding gradient Ricci soliton with $\Ric\leq 0$. Then $S<0$ on $M$.
\end{lma}
\begin{proof}Using $\Ric\leq 0$, we have $|Ric|^2\leq S^2$ and by (\ref{eqn of S})
\begin{eqnarray*}
\D S&=&-S-2|\Ric|^2\\
&\geq&-(1+2S)S.
\end{eqnarray*}
We argue by contradiction. If $S(z)=0$ for some $z$ in $M$, then S attains its interior maximum at $z$. By the strong maximum principle \cite{GilbargTrudinger-2001}, $S\equiv 0$ and hence by (\ref{eqn of S}) $\Ric\equiv 0$. From (\ref{eq-RS-2}), $-2\na^2 f=g$. $g$ is flat by a result in \cite{PigolaRimoldiSetti-2011}, which is absurd.
\end{proof}
With Lemma \ref{Ric<0 implies S<0}, we can finish the proof of Theorem \ref{scalar lower bdd in expander when Ric <0}:
\begin{thm*}Let $(M^n,g,f)$ be an $n$ dimensional complete non-compact gradient expanding Ricci soliton with $n\geq 2$. Suppose that $\lim_{r\to \infty} f=-\infty$ and $(M,g)$ is not flat and has non-positive Ricci curvature. Then there is a positive constant $C$ such that
$$S\leq -Cv^{1-\frac{n}{2}}e^{-v},$$
where $v:=\frac{n}{2}-f$.
\end{thm*}
\begin{proof}%[\textbf{Proof of theorem \ref{scalar lower bdd in expander when Ric <0}:}]
Using the discussion after the proof of Lemma \ref{v proper implies v>0}, we have $v>0$ on $M$. By (\ref{eqn of nav}),
\begin{eqnarray*}
|\na \ln v|^2v&=&\frac{|\na v|^2}{v}\\
&=& 1-\frac{\frac{n}{2}+S}{v}.
\end{eqnarray*}
Using $\Ric\leq 0$ and (\ref{eqn for vS}), we see that $|\Ric|^2\leq S^2$ and
 \begin{eqnarray}\notag
\Delta_{f+2\ln v}(vS)&\geq&-2vS^2-2|\na \ln v|^2vS\\
\label{ineq for vS with Ric<0} &=&-2vS^2-2S+\frac{nS+2S^2}{v}\\
\notag&=&S\Big(-2+\frac{n+2S}{v}-2vS\Big).
 \end{eqnarray}
From $\lim_{r\to\infty} v=\infty$ and (\ref{short form of long computation}), we see that
\begin{eqnarray}\notag
\Delta_{f+2\ln v}(e^{\frac{1}{\sqrt{v}}}v^{2-\frac{n}{2}}e^{-v})&=&e^{\frac{1}{\sqrt{v}}}v^{2-\frac{n}{2}}e^{-v}\Big\{
\big(\frac{1}{2}+o(1)\big)\frac{1}{\sqrt{v}}+S\big(-1+o(1)\big)\Big\}\\
\label{ineq of e-v with Ric<0} &\geq& 0
\end{eqnarray}
near infinity. We now consider a large $R_0$ such that on $M\setminus B_{R_0}(p_0)$, $v>1$ and (\ref{ineq of e-v with Ric<0}) are true, moreover the following hold:
\be\label{contradict ineq 1 with Ric<0}
\frac{-1+\frac{n}{2v}}{1-\frac{1}{v^2}}\leq -\frac{1}{2}
\ee
and
\be\label{contradict ineq 2 with Ric<0}
\frac{1}{2}\geq e^{\frac{1}{\sqrt{v}}}v^{2-\frac{n}{2}}e^{-v}.
\ee
For such $R_0$, there exists a positive constant $\a$ $\in (0,1)$ such that
\be\label{bdry cond with Ric<0}
vS+\a e^{\frac{1}{\sqrt{v}}}v^{2-\frac{n}{2}}e^{-v}<0 \text{  on  } \partial B_{R_0}(p_0).
\ee
We define $Q:=vS+\a e^{\frac{1}{\sqrt{v}}}v^{2-\frac{n}{2}}e^{-v}$. It is easy to see that
$$\limsup_{r\to\infty}Q\leq 0.$$
For any $y$ in $M\setminus \overline{B_{R_0}(p_0)}$ and any $\varepsilon>0$, there is a positive $T>R_0$ such that $y$ $\in$ $B_T(p_0)$ and
\be \label{Q at bdry T with Ric<0}
Q\leq \varepsilon \text{  on  } \partial B_T(p_0).
\ee
Let $\Omega:= B_T(p_0)\setminus\overline{B_{R_0}(p_0)}$ and $z$ $\in$ $\overline{\Omega}$ be a point where $Q$ attains its maximum over $\overline{\Omega}$.
If $z$ $\in$ $\partial \Omega$, then by the boundary conditions (\ref{bdry cond with Ric<0}) and (\ref{Q at bdry T with Ric<0}),
$$Q(y)\leq Q(z)\leq \varepsilon.$$
If $z$ $\in$ $\Omega$, we have by (\ref{ineq for vS with Ric<0}) and (\ref{ineq of e-v with Ric<0}) that at $z$
\begin{eqnarray*}
0&\geq& \Delta_{f+2\ln v} Q\\
&=&\Delta_{f+2\ln v} (vS) +\a \Delta_{f+2\ln v} (e^{\frac{1}{\sqrt{v}}}v^{2-\frac{n}{2}}e^{-v})\\
&\geq& \Delta_{f+2\ln v} (vS)\\
&\geq& S\Big(-2+\frac{n+2S}{v}-2vS\Big).
\end{eqnarray*}
Together with $S<0$, (\ref{contradict ineq 1 with Ric<0}) and (\ref{contradict ineq 2 with Ric<0}), at $z$
\begin{eqnarray*}
vS&\leq& \frac{-1+\frac{n}{2v}}{1-\frac{1}{v^2}}\\
&\leq&-\frac{1}{2}\\
&\leq&-e^{\frac{1}{\sqrt{v}}}v^{2-\frac{n}{2}}e^{-v}\\
&\leq&-\a e^{\frac{1}{\sqrt{v}}}v^{2-\frac{n}{2}}e^{-v}.
\end{eqnarray*}
Hence $Q(z)\leq 0$ and result follows by letting $\varepsilon\to 0$.
\end{proof}
\section{Proof of Theorem \ref{sufficient condition for S geq 0}}
Before moving to the proof of Theorem \ref{sufficient condition for S geq 0}, we recall the statement of the theorem:
\begin{thm*}
Let $(M^n,g,f)$ be a complete noncompact gradient expanding Ricci soliton with dimension $n\geq 3$. We denote the negative part of the scalar curvature $S$ by $S_-$, i.e. $S_-:=(-S)_+:=\max{\{-S, 0\}}$. If $S_-$ is integrable, that is,
\be\label{negative of S is integrable}
\int_M S_- dv_g <\infty,
\ee
then either $S\geq 0$ everywhere or M is isometric to an Einstein manifold with scalar curvature $-\frac{n}{2}$ and finite volume.
\end{thm*}
\begin{proof} W.L.O.G., we may assume $S<0$ somewhere and show that $M$ is Einstein. We are going to prove that $M$ has constant scalar curvature $-\frac{n}{2}$.
By (\ref{eqn of S}) and Cauchy Schwarz inequality,
\begin{eqnarray*}
\D S&=&-S-2|\Ric|^2\\
&\leq&-S-\frac{2}{n}S^2.
\end{eqnarray*}
Let $\delta \in(0,\frac{1}{4})$, $\varepsilon\in (0,1)$ and $\phi$ be any nonnegative compactly supported function, we multiply the above inequality by $(-S)_+(S^2+\varepsilon)^{\delta-\frac{1}{2}}\phi^2\geq 0$ and integrate the inequality over $M$,
\begin{eqnarray*}
(I)+(II)&:=&\int_M (-S)_+(S^2+\varepsilon)^{\delta-\frac{1}{2}}\Delta S\phi^2 -\int_M (-S)_+(S^2+\varepsilon)^{\delta-\frac{1}{2}}\langle\na f, \na S\rangle\phi^2\\
&\leq& \int_M (-S)_+(-S-\frac{2}{n}S^2)(S^2+\varepsilon)^{\delta-\frac{1}{2}}\phi^2\\
&=&\int_{\{S<0\}} (-S)(-S-\frac{2}{n}S^2)(S^2+\varepsilon)^{\delta-\frac{1}{2}}\phi^2\\
&=&\int_{\{S<0\}}\frac{2}{n}S^2(S+\frac{n}{2})(S^2+\varepsilon)^{\delta-\frac{1}{2}}\phi^2.
\end{eqnarray*}
By \cite{GilbargTrudinger-2001}, the weak derivative of $(-S)_+$ is given by
\begin{eqnarray*}
\na (-S)_+&=&\chi_{\{S<0\}}\na (-S)\\
&=&-\chi_{\{S<0\}}\na S,
\end{eqnarray*}
where $\chi_{\{S<0\}}$ is the characteristic function of the set $\{S<0\}$ which is $1$ on $\{S<0\}$ and vanishes elsewhere. Using integration by part, we see that
\begin{eqnarray*}
(I)&=&\int_M (-S)_+(S^2+\varepsilon)^{\delta-\frac{1}{2}}\Delta S\phi^2\\
&=&-\int_M \chi_{\{S<0\}}\la -\na S, \na S\ra(S^2+\varepsilon)^{\delta-\frac{1}{2}}\phi^2-\int_M \chi_{\{S<0\}}(-S)\phi^2\langle \na (S^2+\varepsilon)^{\delta-\frac{1}{2}}, \na S\rangle\\
&&-\int_M \chi_{\{S<0\}}(-S)2\phi\langle \na \phi,\na S\rangle(S^2+\varepsilon)^{\delta-\frac{1}{2}}\\
&\geq&\int_{\{S<0\}} |\na S|^2(S^2+\varepsilon)^{\delta-\frac{1}{2}}\phi^2+(2\delta-1)\int_{\{S<0\}}|\na S|^2(S^2+\varepsilon)^{\delta-\frac{3}{2}}S^2\phi^2\\
&&-\int_{\{S<0\}}  2\phi|\na \phi||\na S||S|(S^2+\varepsilon)^{\delta-\frac{1}{2}}\\
&\geq& 2\delta\int_{\{S<0\}} |\na S|^2(S^2+\varepsilon)^{\delta-\frac{1}{2}}\phi^2-\delta\int_{\{S<0\}} |\na S|^2(S^2+\varepsilon)^{\delta-\frac{1}{2}}\phi^2\\
&&-\frac{1}{\delta}\int_{\{S<0\}} S^2(S^2+\varepsilon)^{\delta-\frac{1}{2}}|\na \phi|^2\\
&=&\delta\int_{\{S<0\}} |\na S|^2(S^2+\varepsilon)^{\delta-\frac{1}{2}}\phi^2-\frac{1}{\delta}\int_{\{S<0\}} S^2(S^2+\varepsilon)^{\delta-\frac{1}{2}}|\na \phi|^2.
\end{eqnarray*}

By (\ref{trace RS eqn}) and integration by part again, we have
\begin{eqnarray*}
(II)&=&-\int_M (-S)_+(S^2+\varepsilon)^{\delta-\frac{1}{2}}\langle\na f, \na S\rangle\phi^2\\
&=&\frac{1}{2\delta+1}\int_M \langle\na f,\na ((-S)_+^2+\varepsilon)^{\delta+\frac{1}{2}}\rangle\phi^2\\
&=&-\frac{1}{2\delta+1}\int_M \Delta f((-S)_+^2+\varepsilon)^{\delta+\frac{1}{2}}\phi^2-\frac{2}{2\delta+1}\int_M \phi\langle \na \phi,\na f\rangle((-S)_+^2+\varepsilon)^{\delta+\frac{1}{2}}\\
&=&\frac{1}{2\delta+1}\int_M (S+\frac{n}{2})((-S)_+^2+\varepsilon)^{\delta+\frac{1}{2}}\phi^2-\frac{2}{2\delta+1}\int_M \phi\langle \na \phi,\na f\rangle((-S)_+^2+\varepsilon)^{\delta+\frac{1}{2}}.
\end{eqnarray*}
%Using $-\frac{n}{2}\leq S <0$ wherever $S<0$, Dominated Convergence theorem and Fat0u lemma
%Since $S\leq 0$, we know that $\{S=0\}\subseteq \{\na S=0\}$.
All in all, we have
\begin{eqnarray*}
0&\leq& \delta\int_{\{S<0\}} |\na S|^2(S^2+\varepsilon)^{\delta-\frac{1}{2}}\phi^2\\
&\leq& \frac{1}{\delta}\int_{\{S<0\}}S^2(S^2+\varepsilon)^{\delta-\frac{1}{2}}|\na \phi|^2+\frac{2}{2\delta+1}\int_M \phi\langle \na \phi,\na f\rangle ((-S)_+^2+\varepsilon)^{\delta+\frac{1}{2}}\\
&&+\int_{\{S<0\}}\frac{2}{n}S^2(S+\frac{n}{2})(S^2+\varepsilon)^{\delta-\frac{1}{2}}\phi^2-\frac{1}{2\delta+1}\int_M (S+\frac{n}{2})((-S)_+^2+\varepsilon)^{\delta+\frac{1}{2}}\phi^2.
\end{eqnarray*}
%Using $|S|\leq\frac{n}{2}$ wherever $S<0$ (see (\ref{general bdd for S in expander})), $(-S)_+=-\chi_{\{S<0\}}S$ and Dominated Convergence theorem, we may let $\varepsilon \to 0$ and get
Using $(-S)_+=-\chi_{\{S<0\}}S$, $\phi$ is compactly supported and Dominated Convergence theorem, we may let $\varepsilon \to 0$ and get
\begin{eqnarray*}
0 &\leq& \frac{1}{\delta}\int_{\{S<0\}}(S^2)^{\delta+\frac{1}{2}}|\na \phi|^2+\frac{2}{2\delta+1}\int_{\{S<0\}} \phi\langle \na \phi,\na f\rangle (S^2)^{\delta+\frac{1}{2}}\\
&&+\int_{\{S<0\}}\frac{2}{n}(S+\frac{n}{2})(S^2)^{\delta+\frac{1}{2}}\phi^2-\frac{1}{2\delta+1}\int_{\{S<0\}} (S+\frac{n}{2})(S^2)^{\delta+\frac{1}{2}}\phi^2\\
&=&\frac{1}{\delta}\int_{\{S<0\}}(S^2)^{\delta+\frac{1}{2}}|\na \phi|^2+\frac{2}{2\delta+1}\int_{\{S<0\}} \phi\langle \na \phi,\na f\rangle (S^2)^{\delta+\frac{1}{2}}\\
&&+\big(\frac{2}{n}-\frac{1}{2\delta+1}\big)\int_{\{S<0\}} (S+\frac{n}{2})(S^2)^{\delta+\frac{1}{2}}\phi^2.
\end{eqnarray*}
Since $n\geq 3$ and $\delta<\frac{1}{4}$, we have $\frac{2}{n}-\frac{1}{2\delta+1}< 0$ and by (\ref{general bdd for S in expander})
\begin{eqnarray}\notag 0&\leq&
\big(\frac{1}{2\delta+1}-\frac{2}{n}\big)\int_{\{S<0\}} (S+\frac{n}{2})(S^2)^{\delta+\frac{1}{2}}\phi^2\\
\label{negative part of S integral is zero}&\leq&\frac{1}{\delta}\int_{\{S<0\}}(S^2)^{\delta+\frac{1}{2}}|\na \phi|^2+\frac{2}{2\delta+1}\int_{\{S<0\}} \phi\langle \na \phi,\na f\rangle (S^2)^{\delta+\frac{1}{2}}.
\end{eqnarray}

%\begin{eqnarray*}
%1+\frac{2}{n}S-\frac{1}{2\delta+1}(S+\frac{n}{2})&=&(\frac{2}{n}-\frac{1}{2\delta+1})S+1-\frac{n}{2(2\delta+1)}\\
%&=&(\frac{2}{n}-\frac{1}{2\delta+1})(S-\frac{\frac{n}{2(2\delta+1)}-1}{\frac{2}{n}-\frac{1}{2\delta+1}})\\
%&=&(\frac{2}{n}-\frac{1}{2\delta+1})(S+\frac{n}{2})\\
%&\leq& 0,
%\end{eqnarray*}
%with equality holds if and only if $S=-\frac{n}{2}$.
Let $R\geq 1$ and $\psi :[0,\infty)\rightarrow \R$ be a smooth real valued function satisfying the following:
$0\leq\psi\leq 1$, $\psi'\leq 0$,
\[ \psi(t)=\begin{cases}
      1 & 0\leq t\leq 1 \\
      0 & 2\leq t\\
   \end{cases}
\]
and
$$|\psi'(t)|\leq c \text{  for all  } t\geq 0.$$
We take the cut off function $\phi(x):=\psi(\frac{r(x)}{R})$, then
\[ \phi=\begin{cases}
      1 & \text{  on  } B_R(p_0)\\
      0 & \text{  on  } M\setminus B_{2R}(p_0)\\
   \end{cases}
\]
and
\be\label{gradient of cut off in exp 3 dim}
|\na \phi|=\frac{|\psi'|}{R}\leq \frac{c}{R}\chi_{B_{2R}\setminus B_R},
\ee
\noindent where $\chi_{B_{2R}\setminus B_R}$ is the characteristic function of the set $B_{2R}(p_0)\setminus B_R(p_0)$ which is $1$ on $B_{2R}(p_0)\setminus B_R(p_0)$ and vanishes elsewhere. We have
\begin{eqnarray*}
\frac{1}{\delta}\int_{\{S<0\}}(S^2)^{\delta+\frac{1}{2}}|\na \phi|^2&\leq& \frac{c^2n^{2\delta}}{\delta4^{\delta}R^2}\int_{\{S<0\}\setminus B_R}|S|\\
&=&\frac{c^2n^{2\delta}}{\delta4^{\delta}R^2}\int_{M\setminus B_R} (-S)_+.
\end{eqnarray*}
Since $|\na f|\leq C(r+1)$ (see (\ref{naf nav sqrtv control linearly in exp})),
\begin{eqnarray*}
\frac{2}{2\delta+1}\int_{\{S<0\}} \phi\langle \na \phi,\na f\rangle (S^2)^{\delta+\frac{1}{2}}&\leq& \frac{2c n^{2\delta}}{(2\delta+1)4^{\delta}}\int_{\{S<0\}\setminus B_R}|S|\\
&=&\frac{2c n^{2\delta}}{(2\delta+1)4^{\delta}}\int_{M\setminus B_R}(-S)_+.
\end{eqnarray*}

By (\ref{negative of S is integrable}), we let $R\to\infty$ and conclude by (\ref{negative part of S integral is zero}) that
$$\int_{\{S<0\}} (S+\frac{n}{2})(S^2)^{\delta+\frac{1}{2}}=0.$$
Hence we may apply (\ref{general bdd for S in expander}) again to get
$$(S+\frac{n}{2})(S^2)^{\delta+\frac{1}{2}}\equiv 0 \text{  on  } \{S<0\}.$$
From this we see that $S=-\frac{n}{2}$ wherever $\{S<0\}$. Since we assume that $\{S<0\}$ is nonempty, we know by connectedness that $S\equiv-\frac{n}{2}$. $M$ is Einstein since
\begin{eqnarray*}
0&=&\D S\\
&=&-S-2|\Ric|^2\\
&=&-\frac{2S}{n}\big(S+\frac{n}{2}\big)-2|\Ric-\frac{S}{n}g|^2.
\end{eqnarray*}
Finiteness of volume is a now consequence of (\ref{negative of S is integrable}).
\end{proof}

\section{Proof of the Theorems \ref{curv estimate in dim 3 exp} and \ref{gap thm for 3 dim exp}}
%The following lemma is well known for experts. For convenience of reader, we include the proof here.
%\begin{lma}Let $(M,g,f)$ be an $n$ dimensional complete non-compact gradient expanding Ricci soliton. Suppose $G$ is a $C^2$ function such that outside a
%compact subset $K$ of $M$,
%\be
%\D G\geq c_0G^2-c_1G-c_2,
%\ee
%for some positive constants $c_i, i=1, 2, 3.$ Then $G$ is bounded from above by a constant $B$ depending only on $c_i$ and $\sup_{K} G$.
%\end{lma}
%\begin{proof}
%let $R\geq 1$ and $\psi :[0,\infty)\rightarrow \R$ be a smooth real valued function satisfying the following:
%$0\leq\psi\leq 1$, $\psi'\leq 0$,
%\[ \psi(t)=\begin{cases}
%      1 & 0\leq t\leq 1 \\
%      0 & 2\leq t\\
%   \end{cases}
%\]
%and
%$$|\psi''(t)|+|\psi'(t)|\leq c \text{  for all  } t\geq 0.$$
%\end{proof}
%\noindent
In this section, we study the geometry of three dimensional gradient expander. For the convenience of reader, we recall the statement of Theorem \ref{curv estimate in dim 3 exp}\ref{3 dim S control Rm exp}.
\begin{thm*}Let $(M,g,f)$ be a $3$ dimensional complete non-compact gradient expanding Ricci soliton with bounded non-negative scalar curvature. Then the curvature tensor $\Rm$ is bounded, moreover
$$|\Rm|\leq c\sqrt{S} \text{  on  } M,$$
for some constant $c>0$.
\end{thm*}
\begin{proof}Throughout this proof, we use $c$ to denote an absolute constant, its value may be different from line by line. W.L.O.G, we may assume $M$ is not flat, or else we have nothing to prove. By the strong minimum principle and (\ref{eqn of S}), $S$ is positive on $M$. Since the Weyl tensor is zero in dimension three, it suffices to bound the Ricci tensor.
By (\ref{eqn of Rc in exp})
\begin{eqnarray}\label{ineq for norm of RC sq in exp}
\notag\D |\Ric|^2 &=&2|\na \Ric|^2+2\langle\Ric, \D \Ric\rangle\\
&=&2|\na \Ric|^2-2|\Ric|^2-4R_{ij}R_{iklj}R_{kl}\\
\notag&\geq&2|\na \Ric|^2-2|\Ric|^2-c|\Ric|^3,
\end{eqnarray}
for some absolute constant $c$. Using (\ref{eqn of S}),
\begin{eqnarray}\label{ineq for inv S in exp}
\notag\D S^{-1}&=&-S^{-2}\D S+2S^{-1}|\na \ln S|^2\\
&=&-S^{-2}(-S-2|\Ric|^2)+2S^{-1}|\na \ln S|^2\\
\notag&=&S^{-1}+2S^{-2}|\Ric|^2+2S^{-1}|\na \ln S|^2.
\end{eqnarray}
Hence by Kato's inequality, (\ref{ineq for norm of RC sq in exp}) and (\ref{ineq for inv S in exp}),
\begin{eqnarray}\label{ineq for norm of Ric times invS in exp}
\notag\D (S^{-1}|\Ric|^2)&=&S^{-1}\D |\Ric|^2+2\langle\na S^{-1}, \na |\Ric|^2\rangle+|\Ric|^2\D S^{-1}\\
\notag&\geq&2S^{-1}|\na \Ric|^2-2S^{-1}|\Ric|^2-cS^{-1}|\Ric|^3\\
&&-4S^{-1}|\na \ln S||\Ric||\na \Ric|+S^{-1}|\Ric|^2\\
\notag&&+2S^{-2}|\Ric|^4+2S^{-1}|\na \ln S|^2|\Ric|^2.
\end{eqnarray}
By completing square, we see that
\begin{eqnarray*}
2S^{-1}|\na \Ric|^2-4S^{-1}|\na \ln S||\Ric||\na \Ric|&=&2S^{-1}\Big(|\na\Ric|-|\na \ln S||\Ric|\Big)^2\\
&&-2S^{-1}|\na \ln S|^2|\Ric|^2\\
&\geq&-2S^{-1}|\na \ln S|^2|\Ric|^2.\\
\end{eqnarray*}
\begin{eqnarray*}
\notag\D (S^{-1}|\Ric|^2)&\geq&2S^{-2}|\Ric|^4-cS^{-1}|\Ric|^3-S^{-1}|\Ric|^2.
\end{eqnarray*}
Let $u:=S^{-1}|\Ric|^2$, the above differential inequality can be rewritten as
\begin{eqnarray}
\notag\D u&\geq& 2u^2-cS^{\frac{1}{2}}u^{\frac{3}{2}}-u\\
\label{reduced ineq for norm of RC sq in exp}&\geq& u^2-c(1+S)u\\
\notag &\geq& u^2-c_0u,
\end{eqnarray}
where $c_0$ is a positive constant depending on the global upper bound of the scalar curvature.
Let $R\geq 1$ and $\psi :[0,\infty)\rightarrow \R$ be a smooth real valued function satisfying the following:
$0\leq\psi\leq 1$, $\psi'\leq 0$,
\[ \psi(t)=\begin{cases}
      1 & 0\leq t\leq 1 \\
      0 & 2\leq t\\
   \end{cases}
\]
and
$$|\psi''(t)|+|\psi'(t)|\leq c \text{  for all  } t\geq 0.$$
We consider a function $\phi(x):=\psi(\frac{r(x)}{R})$, then
\[ \phi=\begin{cases}
      1 & \text{  on  } B_R(p_0)\\
      0 & \text{  on  } M\setminus B_{2R}(p_0)\\
   \end{cases}
\]
and
\be\label{gradient of cut off in exp 3 dim}
|\na \phi|=\frac{|\psi'|}{R}\leq \frac{c}{R}.
\ee
Using the Laplacian comparison theorem in \cite{WeiWylie-2009} (see Theorem 3.1), there exists a positive constant $\beta$ (independent on $R$ and $x\in M\setminus B_1(p_0)$) such that on $M\setminus B_1(p_0)$
\be\label{lap comp in exp}
\D r(x)\leq \frac{1}{2}r(x)+\beta.
\ee
Hence
\begin{eqnarray}
\notag\D \phi&=&\frac{\psi'}{R}\D r+\frac{\psi''}{R^2}|\na r|^2\\
\label{ineq for D cut off in exp 3 dim}&\geq&-\frac{c}{R}\big(R+\beta\big)-\frac{c}{R^2}\\
\notag&\geq&-c-\frac{c(1+\beta)}{R}.
\end{eqnarray}
Let $G$ be $\phi^2u$, we have by (\ref{reduced ineq for norm of RC sq in exp}), (\ref{gradient of cut off in exp 3 dim}) and (\ref{ineq for D cut off in exp 3 dim})
\begin{eqnarray}
\notag\phi^2\D G&=& \phi^4\D u+4\phi^3\langle\na \phi, \na(\phi^2u\phi^{-2})\rangle +(2\phi\D\phi+2|\na\phi|^2)G\\
\label{ineq for G in exp 3 dim}&\geq& G^2-c_0G+4\phi\langle\na \phi,\na G\rangle+ (2\phi\D \phi-6|\na \phi|^2)G\\
\notag&\geq& G^2-\big(c_0+c+\frac{c(1+\beta)}{R}\big)G+4\phi\langle\na \phi,\na G\rangle.
\end{eqnarray}
Suppose $G$ attains its maximum at $q$. If $q\in \overline{B_1(p_0)}$, then
$$G\leq G(q)\leq \sup_{\overline{B_1(p_0)}}u,$$
R.H.S. is independent on $R\geq 1$. If $q\in M\setminus\overline{B_1(p_0)}$, then by the maximum principle, we have by (\ref{ineq for G in exp 3 dim}),
$$0\geq G^2(q)-\big(c_0+c+\frac{c(1+\beta)}{R}\big)G(q).$$
In either cases, we get the following bound for $G$:
$$G\leq \sup_{\overline{B_1(p_0)}}u+c_0+c+\frac{c(1+\beta)}{R},$$
result then follows by letting $R\to \infty$.
\end{proof}
For general gradient expander, the scalar curvature may be negative. Nonetheless, $\Rm$ is still bounded if the scalar curvature is bounded.
\begin{thm}\label{no sign bdd S implies bdd Rm in exp}
Let $(M^3, g, f)$ be a $3$ dimensional complete noncompact gradient expanding Ricci soliton. Then there exist positive constant $C_1$ and $C_2$ such that on $M$
\be
|\Rm|\leq C_1S+C_2.
\ee
In particular the curvature tensor is bounded If in addition the scalar curvature $S$ is bounded.
%If in addition the scalar curvature $S$ is bounded, then the curvature tensor is also bounded.
\end{thm}
\begin{proof}Since in dimension 3, the curvature tensor is controlled by the Ricci tensor, we shall estimate the Ricci tensor. Using the computation (\ref{ineq for norm of RC sq in exp}) in the proof of Theorem \ref{curv estimate in dim 3 exp}\ref{3 dim S control Rm exp} and Kato's inequality, we have
$$\D|\Ric|\geq -|\Ric|-c|\Ric|^2$$
for some absolute constant $c>0$ wherever $|\Ric|>0$. Hence by (\ref{eqn of S}) and (\ref{general bdd for S in expander})
\begin{eqnarray}\notag
\D (|\Ric|-AS) &\geq& (2A-c)|\Ric|^2-|\Ric|+AS\\
\label{max prin argu for curva in 3 dim in exp without any condition}&\geq& |\Ric|^2-|\Ric|+AS\\
\notag&\geq& |\Ric|^2-|\Ric|-\frac{3A}{2},
\end{eqnarray}
for all large positive constant $A$. Let $u:= |\Ric|-AS$ and $G:=\phi^2u$. For any $R\geq 1$, we consider the cutoff function $\phi(x)=\psi(\frac{r(x)}{R})$ as in (\ref{gradient of cut off in exp 3 dim}), (\ref{lap comp in exp}) and (\ref{ineq for D cut off in exp 3 dim}). Using (\ref{general bdd for S in expander}) and (\ref{max prin argu for curva in 3 dim in exp without any condition}), we compute directly as in (\ref{ineq for G in exp 3 dim}) on the set where $G$ is positive to get
\begin{eqnarray}\notag
\phi^2\D G&=&\phi^4\D u+4\phi\langle\na \phi, \na G\rangle+(2\phi\D \phi-6|\na \phi|^2)G\\
\notag
&\geq&\phi^4\D u-(c+\frac{c(1+\beta)}{R})G+4\phi\langle\na \phi, \na G\rangle\\
\label{ineq for G without bdd curv in dim 3}&\geq&\phi^4\D u-(c+\frac{c(1+\beta)}{R})\phi^2|\Ric|-\frac{3Ac}{2}(1+\frac{(1+\beta)}{R})+4\phi\langle\na \phi, \na G\rangle\\
\notag
&\geq&\phi^4|\Ric|^2-(1+c+\frac{c(1+\beta)}{R})\phi^2|\Ric|-\frac{3A}{2}(1+c+\frac{c(1+\beta)}{R})\\
\notag&&+4\phi\langle\na \phi, \na G\rangle,
\end{eqnarray}
where $\beta$ is the positive constant in (\ref{lap comp in exp}). Suppose $G$ attains its maximum at $q$. If $q\in \overline{B_1(p_0)}$, then
$$G\leq G(q)\leq \sup_{\overline{B_1(p_0)}}(|\Ric|+A|S|),$$
R.H.S. is independent on $R\geq 1$. If $q\in M\setminus\overline{B_1(p_0)}$, we may assume $G(q)>0$. Hence $|\Ric|(q)>0$ and $u$ is smooth near $q$. By the maximum principle, (\ref{general bdd for S in expander}) and (\ref{ineq for G without bdd curv in dim 3}), we have
\begin{eqnarray*}
G&\leq& G(q)\\
&=&\phi^2(q)|\Ric|(q)-\phi^2(q)AS(q)\\
&\leq&\phi^2(q)|\Ric|(q)+\frac{3A}{2}\\
&\leq&1+c+\frac{c(1+\beta)}{R}+\sqrt{\frac{3A}{2}\Big(1+c+\frac{c(1+\beta)}{R}\Big)}+\frac{3A}{2}.
\end{eqnarray*}
Results follows by letting $R\to\infty$.
\end{proof}

To prove Theorem \ref{curv estimate in dim 3 exp}\ref{3 dim S decay implies Rm decay exp}, we need to control the change in distance along the flow of $\na f$. Recall that it is a result of Zhang \cite{Zhang-2009} that the flow of $\na f$ exists for all time $t\in \R$.
\begin{lma}\label{dist estimate along grad flow in exp} Let $(M^n, g, f)$ be an $n$ dimensional complete gradient expanding Ricci soliton and $\phi_s$ be the flow of $\na f$ with $\phi_0$ being the identity map. Then for any $x \in M$ and $s\geq 0$,
\be
r(\phi_s(x))\leq e^{\frac{s}{2}}(r(x)+2\sqrt{v(p_0)}),
\ee
where $v=\frac{n}{2}-f$ and $p_0$ is the base point of the distance function $r$.
\end{lma}
\begin{proof}
%By (\ref{general bdd for S in expander}) and (\ref{eqn of nav}), $|\na \sqrt{v+\varepsilon}|\leq \frac{1}{2}$ for any $\varepsilon>0$ and
%\begin{eqnarray*}
%|\na f|&=&|\na v|\\
%&\leq& \sqrt{v+\varepsilon}\\
%&\leq& \frac{1}{2}r+\sqrt{v(p_0)+\varepsilon}.
%\end{eqnarray*}
%Let $\varepsilon\to 0$,
It follows from (\ref{naf nav sqrtv control linearly in exp}) that
\be\label{est of na f in exp}
|\na f|\leq \frac{1}{2}r+\sqrt{v(p_0)}.
\ee
For simplicity, we denote $r(\phi_s(x))$ by $r_s$, $\sqrt{v(p_0)}$ by $c_1$. By triangle inequality and (\ref{est of na f in exp}),
\begin{eqnarray}\notag
r_s&\leq& r_0+\int_0^s|\dot{\phi}_{\tau}(x)|d\tau\\
\label{rs estimate in distance estimate in exp}&=& r_0+\int_0^s|\na f(\phi_{\tau}(x))|d\tau\\
\notag&\leq&r_0+\int_0^s\frac{1}{2}r_{\tau}+c_1d\tau.
\end{eqnarray}
Let $w(s):=r_0+\int_0^s\frac{1}{2}r_{\tau}+c_1d\tau+2c_1$, the above inequality can be rewritten as
$$w'\leq \frac{w}{2}.$$
Integrating the inequality w.r.t. $s$, we have by (\ref{rs estimate in distance estimate in exp})
\begin{eqnarray*}
r_s&\leq& w(s)\\
&\leq&e^{\frac{s}{2}}(r_0+2c_1)\\
&=& e^{\frac{s}{2}}(r(x)+2\sqrt{v(p_0)}).
\end{eqnarray*}
\end{proof}
Theorem \ref{curv estimate in dim 3 exp}\ref{3 dim S decay implies Rm decay exp} is now a consequence of Theorem \ref{no sign bdd S implies bdd Rm in exp} and the following theorem. The proof is motivated by the arguments in \cite{DengZhu-2018} and \cite{MunteanuWang-2017}.
\begin{thm}\label{n dim S control Ric if Rm bdd exp} Let $(M^n, g, f)$ be an $n$ dimensional complete noncompact gradient expanding Ricci soliton with bounded curvature. If $\lim_{x\to\infty} S(x)=0$, then
$$\lim_{x\to\infty}|\Ric|(x)=0.$$
\end{thm}
\begin{proof} We argue by contradiction. Suppose the claim is not true. Then there exist a sequence $y_k\in M$ $\to \infty$ as $k\to\infty$ and $\varepsilon_0$ $\in (0,1)$ such that for all $k$
\be\label{Pk lower bdd in exp}
P_k:=|\Ric(g)|(y_k)\geq \varepsilon_0.
\ee
Since $|\Rm|$ is assumed to be bounded on $M$, we see that the derivatives of $\Rm$ of any orders are also bounded by the Shi's derivative estimate (see \cite{ChowLuNi-2006}). For any nonnegative integer $l$, we may let $L_l:=\sup_M|\na^l \Rm(g)|<\infty$. Let $g(t)$ be the Ricci flow generated by the soliton, i.e. $g(t)=(1+t)\psi_t^*g$ for all $t>-1$, where $\psi_t$ is the flow of the non-autonomous vector field $\frac{\na f}{1+t}$ with $\psi_0$ being the identity map. We are going to rescale $g(t)$ near $y_k$. We consider the flow $g_k(t):=P_k g(\frac{t}{P_k})=(P_k+t)\psi_{\frac{t}{P_k}}^*g$. Then all $t\in [-\frac{\varepsilon_0}{10}, \frac{\varepsilon_0}{10}]$ and $z\in M$
\begin{eqnarray*}
|\Rm(g_k(t))|_{g_k(t)}(z)&=&\frac{1}{P_k}|\Rm((1+\frac{t}{P_k})\psi_{\frac{t}{P_k}}^*g)|_{(1+\frac{t}{P_k})\psi_{\frac{t}{P_k}}^*g}(z)\\
&=&\frac{1}{P_k+t}|\Rm(g)|_g(\psi_{\frac{t}{P_k}}(z))\\
&\leq&\frac{10}{9\varepsilon_0}L_0.
\end{eqnarray*}
Similarly for all the derivatives of the curvature $\Rm$,
\begin{eqnarray*}
|\na ^l_{g_k(t)}\Rm(g_k(t))|_{g_k(t)}(z)&=&(P_k+t)|\na ^l_{g_k(t)}\Rm(\psi_{\frac{t}{P_k}}^*g)|_{g_k(t)}(z)\\
&=&(P_k+t)|\na ^l_{\psi_{\frac{t}{P_k}}^*g}\Rm(\psi_{\frac{t}{P_k}}^*g)|_{g_k(t)}(z)\\
&=& \frac{1}{(P_k+t)^{\frac{l+2}{2}}}|\na ^l_{\psi_{\frac{t}{P_k}}^*g}\Rm(\psi_{\frac{t}{P_k}}^*g)|_{\psi_{\frac{t}{P_k}}^*g}(z)\\
&=& \frac{1}{(P_k+t)^{\frac{l+2}{2}}}|\na^l_g\Rm(g)|_g(\psi_{\frac{t}{P_k}}(z))\\
&\leq&\Big(\frac{10}{9\varepsilon_0}\Big)^{\frac{l+2}{2}}L_l.
\end{eqnarray*}
For all $k\in\mathbb{N}$, let $\iota_k:\R^n\to T_{y_k}M$ (with metric $g_k(0)$) be any linear isometry such that $\iota_k(0)=0$ and $F_k:=\ex^{g_k(0)}_{y_k}\circ\iota_k$, where $\ex^{g_k(0)}_{y_k}$ denotes the exponential map w.r.t. $g_k(0)$ at $y_k$. By Rauch comparison theorem, $F_k$ is a local diffeomorphism on $\{x\in \R^n: |x|<\sqrt{\frac{9\pi^2\varepsilon_0}{10L_0}}\}$. By Hamilton compactness theorem (see \cite{Hamilton-1995.2}, \cite{Chowetal-2007} and the proof of Lemma 4.4 in \cite{DengZhu-2018}), there exist positive $\delta$ (depending only on $\varepsilon_0$, $L_0$ and $n$) and subsequence $k_j$ such that as $j$ $\to$ $\infty$
\be
\big(B_{\delta}(0), F_{k_j}^*g_{k_j}(t)\big)\to \big(B_{\delta}(0), h_{\infty}(t)\big)
\ee
in $C^{\infty}_{loc}$ sense on $B_{\delta}(0)\times (-\frac{\varepsilon_0}{10}, \frac{\varepsilon_0}{10})$, where $B_{\delta}(0):=\{x\in\R^n: |x|<\delta\}$ and $h_{\infty}(t)$ is a solution to the Ricci flow. Moreover, by the local smooth convergence of the metric, we have
\begin{eqnarray*}
|\Ric(h_{\infty}(0))|_{h_{\infty}(0)}(0)&=&\lim_{j\to\infty}|\Ric(g_{k_j}(0))|_{g_{k_j}(0)}(y_{k_j})\\
&=&\lim_{j\to\infty}\frac{1}{P_{k_j}}|\Ric(g)|_g(y_{k_j})\\
&=&\lim_{j\to\infty}1\\
&=&1.
\end{eqnarray*}
Hence $h_{\infty}(0)$ is not Ricci flat. We claim that $h_{\infty}(t)$ is scalar flat for all $t\in (-\frac{\varepsilon_0}{10},0]$. We first assume the claim and prove the theorem. By the evolution equation of the scalar curvature along the Ricci flow (see \cite{ChowLuNi-2006})
$$2|\Ric(h_{\infty}(t))|^2_{h_{\infty}(t)}=\frac{\partial}{\partial t}S_{h_{\infty}(t)}-\Delta_{{h_{\infty}(t)}}S_{h_{\infty}(t)}=0,$$
which is impossible. It remains to justify our claim, i.e. $h_{\infty}(t)$ is scalar flat. For all $(z, t)\in B_{\delta}(0)\times(-\frac{\varepsilon_0}{10},0],$
the scalar curvature with respect to the metric $h_{\infty(t)}$ satisfies
\begin{eqnarray}\notag
S_{h_{\infty(t)}}(z)&=&\lim_{j\to\infty}S_{F_{k_j}^*g_{k_j}(t)}(z)\\
\label{scalar of the limiting soln in exp}&=&\lim_{j\to\infty}\frac{1}{(P_{k_j}+t)}S_{\psi_{\frac{t}{P_{k_j}}}^*g}(F_{k_j}(z))\\
\notag&=&\lim_{j\to\infty}\frac{1}{(P_{k_j}+t)}S_g(\psi_{\frac{t}{P_{k_j}}}\circ F_{k_j}(z)).
\end{eqnarray}
By the assumption $\lim_{x\to\infty}S(x)=0$, we are done if $\lim_{j\to\infty}\psi_{\frac{t}{P_{k_j}}}(F_{k_j}(z))=\infty$. Since $r_k:=d(y_k,p_0)\to\infty$ as $k\to\infty$, there is a  $N_0\in \mathbb{N}$ such that for all $j\geq N_0$,
\be\label{contradict ineq for decay of Rm in exp}
r_{k_j}>\sqrt{\frac{10}{9}}\Big[\sqrt{r_{k_j}}+2\sqrt{v(p_0)}\Big]+\frac{\delta}{\sqrt{\varepsilon_0}}.
\ee
We are going to show that for all $j\geq N_0$, $z\in B_{\delta}(0)$ and $t\in (-\frac{\varepsilon_0}{10},0]$
\be\label{dist estimate for limiting RF in exp}
d(\psi_{\frac{t}{P_{k_j}}}(F_{k_j}(z)), p_0)\geq \sqrt{r_{k_j}}.
\ee
Assume by contradiction that it is not true. Then $d(\psi_{\frac{t}{P_{k_j}}}(F_{k_j}(z)), p_0)< \sqrt{r_{k_j}}$ for some $j\geq N_0$, $z\in B_{\delta}(0)$ and $t\in (-\frac{\varepsilon_0}{10},0]$. We consider the flow $\phi_s$ of $\na f$ with $\phi_0$ being the identity map, it is related to $\psi_t$ in the following way
$$\psi_t=\phi_{\ln(1+t)} \text{  for all  } t>-1.$$
By Lemma \ref{dist estimate along grad flow in exp} and $\phi_{-\ln(1+\frac{t}{P_{k_j}})}\circ \psi_{\frac{t}{P_{k_j}}}(F_{k_j}(z))=F_{k_j}(z)$,
\begin{eqnarray}\notag
d(F_{k_j}(z),p_0)&\leq&  \frac{1}{\big(1+\frac{t}{P_{k_j}}\big)^{\frac{1}{2}}}\Big[d(\psi_{\frac{t}{P_{k_j}}}(F_{k_j}(z)) ,p_0)+2\sqrt{v(p_0)}\Big]\\
\label{dist estimate for Fkz}&\leq&\sqrt{\frac{10}{9}}\Big[\sqrt{r_{k_j}}+2\sqrt{v(p_0)}\Big],
\end{eqnarray}
we used $\frac{t}{P_k}\geq -\frac{\varepsilon_0}{10 P_k}\geq -\frac{1}{10}$ in the last inequality. By the definition of $F_k:=\ex^{g_k(0)}_{y_k}\circ\iota_k$,
$$d_{g_{k_j}(0)}(F_{k_j}(z), y_{k_j})=\sqrt{P_{k_j}}d_g(F_{k_j}(z), y_{k_j})\leq \delta.$$
Furthermore by triangle inequality and (\ref{dist estimate for Fkz}),
\begin{eqnarray*}
r_{k_j}-\frac{\delta}{\sqrt{P_{k_j}}}&\leq& d_g(F_{k_j}(z),p_0)\\
&\leq&\sqrt{\frac{10}{9}}\Big[\sqrt{r_{k_j}}+2\sqrt{v(p_0)}\Big],
\end{eqnarray*}
which is absurd by (\ref{Pk lower bdd in exp}) and (\ref{contradict ineq for decay of Rm in exp}). We proved that (\ref{dist estimate for limiting RF in exp}) holds. $S_{h_{\infty}(t)}\equiv 0$ now follows from (\ref{Pk lower bdd in exp}), (\ref{scalar of the limiting soln in exp}) and (\ref{dist estimate for limiting RF in exp}).

\end{proof}
\begin{remark}\label{derivative of Ric goes to 0}
By choosing $P_k=\big(|\na ^l \Ric|(y_k)\big)^{\frac{2}{l+2}}$ in \eqref{Pk lower bdd in exp}, where $l$ is any positive integer, we can use a similar argument to show that $\lim_{x\to\infty}|\na ^l \Ric|=0$ under the same assumptions of Theorem \ref{n dim S control Ric if Rm bdd exp}.
\end{remark}

\begin{proof}[\textbf{Proof of Theorem \ref{curv estimate in dim 3 exp}\ref{3 dim S decay implies Rm decay exp}:}]
By Theorem \ref{no sign bdd S implies bdd Rm in exp}, $M$ has bounded curvature. We then apply Theorem \ref{n dim S control Ric if Rm bdd exp} to conclude that $\Ric\to 0$ as $x\to \infty$. Since the Weyl tensor is zero in dimension 3, we have $\lim_{x\infty} |\Rm|=0$.
\end{proof}

Corollary \ref{proper f in 3 dim} is a consequence of the following proposition and Theorem \ref{curv estimate in dim 3 exp}\ref{3 dim S decay implies Rm decay exp}. The proposition is known and we give a proof for the sake of completeness.
\begin{prop}\label{v equiv r2 under Ric lower bdd decay to 0}Let $(M^n, g, f)$ be an $n$ dimensional complete noncompact gradient expanding Ricci soliton. If $\liminf_{x\to\infty}\Ric\geq 0$, then
$$\lim_{x\to \infty} \frac{4v(x)}{r^2(x)}=1,$$
where $v=\frac{n}{2}-f$.
\end{prop}

% We are going to show that $\lim_{x\to\infty}\frac{4v(x)}{r^2(x)}=1$.
%$By (\ref{general bdd for S in expander}) and (\ref{eqn of nav}), $|\na \sqrt{v+\delta}|\leq \frac{1}{2}$ for any $\delta>0$ and
%$$$\sqrt{v(x)+\delta}\leq \frac{1}{2}r(x)+\sqrt{v(p_0)+\delta}.$$
%$By letting $\delta\to 0$,

\begin{proof}
It can be seen from (\ref{naf nav sqrtv control linearly in exp}) that $\overline{\lim}_{x\to\infty}\frac{v(x)}{r^2(x)}\leq \frac{1}{4}$. Since $\liminf_{x\to\infty}\Ric\geq 0$, for any small positive number $\varepsilon$, $\exists$ $R_0\gg 1$ such that on $M\setminus B_{R_0}(p_o)$
\begin{eqnarray*}
\na^2v&=&-\na^2 f\\
&=&\Ric+\frac{1}{2}g\\
&\geq&(\frac{1}{2}-\varepsilon)g.
\end{eqnarray*}
For any $x$ in $M\setminus B_{R_0}(p_o)$, consider a normalized minimizing geodesic $\gamma$ joining $p_0$ to $x$. We integrate the above inequality along the geodesic to get
$$\langle\na v, \dot{\gamma}\rangle(\gamma(t))-\langle\na v, \dot{\gamma}\rangle(\gamma(R_0))\geq (\frac{1}{2}-\varepsilon)(t-R_0),$$
where $t\geq R_0$. By integrating the inequality w.r.t. $t$, we see that
\begin{eqnarray*}
v(x)-v(\gamma(R_0))&\geq& (\frac{1}{4}-\frac{\varepsilon}{2})(r(x)-R_0)^2+\langle\na v, \dot{\gamma}\rangle(\gamma(R_0))(r(x)-R_0)\\
&\geq&(\frac{1}{4}-\frac{\varepsilon}{2})(r(x)-R_0)^2-\sup_{\overline{B_{R_0}(p_o)}}|\na v|(r(x)-R_0).
\end{eqnarray*}
Hence it is clear that $\underline{\lim}_{x\to\infty}\frac{v(x)}{r^2(x)}\geq \frac{1}{4}-\frac{\varepsilon}{2}$. We let $\varepsilon\to 0$ and conclude that $\lim_{x\to\infty}\frac{4v}{r^2}=1$.
\end{proof}

\begin{lma}\label{r2Rm to 0} Let $(M^3,g,f)$ be a $3$ dimensional complete non-compact gradient expanding Ricci soliton. If $\displaystyle\lim_{x\to \infty} r^2(x)S(x)=0$, then $\displaystyle\lim_{x\to \infty} r^2(x)|\Rm|(x)=0$.
\end{lma}
\begin{proof} By Theorem \ref{curv estimate in dim 3 exp}\ref{3 dim S decay implies Rm decay exp}, $|\Rm|=o(1)$. Using local Shi's estimate (see Lemma 2.6 in \cite{Deruelle-2017}), there exists a positive constant $C$ such that for all $p\in$ $M$ and $R\geq 1$
\be\label{Shi estimate expander}
|\na \Rm|(p)\leq C\sup_{B_R(p)}|\Rm|\left[1+\sup_{B_R(p)}|\Rm|+\frac{\sup_{B_R(p)\setminus B_{\frac{R}{2}}(p)}|\na f|}{R}\right]^{\frac{1}{2}}.
\ee
Hence $|\na \Rm|=o(1)$. By Corollary \ref{proper f in 3 dim}, (\ref{eqn of naf}) and (\ref{naf nav sqrtv control linearly in exp}),
\be\label{na f equiv r in dim 3 exp}
C^{-1}r\leq |\na f| \leq Cr
\ee
near infinity for some constant $C>0$. Let $\nu:=\frac{\na f}{|\na f|}$. In dimension three, the Weyl tensor vanishes and (see also \cite{Deruelle-2017})
\be\label{weyl zero in dim 3}
R_{ij}=R_{\nu ij \nu}+ \frac{S}{2}(g_{ij}- g_{i\nu}g_{j\nu})- R_{\nu\nu}g_{ij}+ R_{j\nu}g_{i\nu}+ R_{i\nu}g_{j\nu}.
\ee
To proceed, we need the following identity for gradient Ricci solitons which follows from Ricci identity and (\ref{eq-RS-2}):
\be
R_{ij,k}-R_{ik,j}=R_{kjil}f_l.
\ee
By (\ref{na f equiv r in dim 3 exp}), we see that
\begin{eqnarray}\label{curvature in normal direction}
R_{ijk\nu}&=&\frac{R_{kj,i}-R_{ki,j}}{|\na f|}\\
          \notag &=& o(r^{-1}).
\end{eqnarray}
By (\ref{weyl zero in dim 3}), It is now clear that $|\Rm|\leq c|\Ric|=o(r^{-1})$. By (\ref{Shi estimate expander}) and (\ref{curvature in normal direction}) again, we have $|\na \Rm|=o(r^{-1})$ and $R_{ijk\nu}=o(r^{-2})$. $|\Rm|=o(r^{-2})$ then follows from (\ref{weyl zero in dim 3}).
\end{proof}

%\begin{lma}\label{r2Rm to 0} Let $(M,g,f)$ be a $3$ dimensional complete non-compact gradient expanding Ricci soliton with non-negative scalar curvature. If $\displaystyle\lim_{x\to \infty} r^2(x)S(x)=0$, then $\displaystyle\lim_{x\to \infty} r^2(x)|\Rm|(x)=0$.
%\end{lma}
%\begin{proof} By Theorem \ref{3 dim S control Rm exp}, $|\Rm|=o(r^{-1})$. Using local Shi's estimate (see Lemma 2.6 in \cite{Deruelle-2017}), there exists a positive constant $C$ such that for all $p\in$ $M$ and $R\geq 1$
%\be\label{Shi estimate expander}
%|\na \Rm|(p)\leq C\sup_{B_R(p)}|\Rm|\left[1+\sup_{B_R(p)}|\Rm|+\frac{\sup_{B_R(p)\setminus B_{\frac{R}{2}}(p)}|\na f|}{R}\right]^{\frac{1}{2}}.
%\ee
%Hence by Corollary \ref{proper f in 3 dim} and (\ref{eqn of naf}),
%\be\label{na f equiv r in dim 3 exp}
%C^{-1}r\leq |\na f| \leq Cr
%\ee
%near infinity for some constant $C>0$ and $|\na \Rm|=o(r^{-1})$. Let $\nu:=\frac{\na f}{|\na f|}$. In dimension three, the Weyl tensor vanishes and (see also \cite{Deruelle-2017})
%\be\label{weyl zero in dim 3}
%R_{ij}=R_{\nu ij \nu}+ \frac{S}{2}(g_{ij}- g_{i\nu}g_{j\nu})- R_{\nu\nu}g_{ij}+ R_{j\nu}g_{i\nu}+ R_{i\nu}g_{j\nu}.
%\ee
%To proceed, we need the following identity for gradient Ricci solitons which follows from Ricci identity and (\ref{eq-RS-2}):
%\be
%R_{ij,k}-R_{ik,j}=R_{kjil}f_l.
%\ee
%By \ref{na f equiv r in dim 3 exp}, we see that
%\begin{eqnarray*}
%R_{ijk\nu}&=&\frac{R_{ij,k}-R_{ik,j}}{|\na f|}\\
%          &=& o(r^{-2}).
%\end{eqnarray*}
%By (\ref{weyl zero in dim 3}), It is now clear that $|\Rm|\leq c|\Ric|=o(r^{-2})$.
%\end{proof}
Before we move on, let us recall some basic definitions:
\begin{definition}(\cite{Schoen-1989}, \cite{SchoenYau-2017}) An $n$ dimensional complete Riemannian manifold $M$ is asymptotically flat if there exist $R_0>0$, compact set $K\subseteq M$ and diffeomophism $\psi: M\setminus K \rightarrow \R^{n}\setminus \{|x|\leq R_0\}$ such that in this coordinate
$$|x|^p|g_{ij}-\delta_{ij}|+|x|^{1+p}|\partial g_{ij}|+|x|^{2+p}|\partial^2 g_{ij}|\leq C \text{  as } x\to \infty,$$
for some $p>\frac{n-2}{2}$ and $C>0$, where $\partial$ denotes the partial derivative. We also require the scalar curvature $S=O(|x|^{-q})$ for some $q>n$.
\end{definition}
Let $X$ be any smooth $n-1$ dimensional closed manifold with Riemannian metric $g_X$, $C(X)$ is defined to be the cone over $X$, i.e. $\{(t,\omega):$  $t>0, \omega \in X\}$. $g_C$ and $\na_C$ denote the metric $dt^2+t^2g_X$ on $C(X)$ and its Riemannian connection respectively. $\overline{B(o, R)}\subseteq C(X)$ is the set given by $\{(t,\omega):$  $R\geq t>0, \omega \in X\}$.
With the above preparations, we are going to prove Theorem \ref{gap thm for 3 dim exp}:
\begin{thm*} Let $(M,g,f)$ be a $3$ dimensional complete non-compact gradient expanding Ricci soliton. Suppose that
$$\displaystyle\lim_{x\to \infty} r^2(x)S(x)=0.$$
Then $M$ is isometric to $\R^{3}$.
\end{thm*}
\begin{proof}%[\textbf{Proof of Theorem \ref{gap thm for 3 dim expander}:}]
By Theorem 1.3 in \cite{Deruelle-2017} and Lemma \ref{r2Rm to 0}, the curvature tensor $\Rm$ satisfies
%\be \label{cone derivative estimates for Rm}
%|\na^k \Rm|_g=O(v^{1+\frac{k-n}{2}}e^{-v}) \text{  as  } r \to \infty, \forall k \in \mathbb{N}.
%\ee
\be \label{cone derivative estimates for Rm}
|\na^k \Rm|_g=O(v^{1+\frac{k-3}{2}}e^{-v}) \text{  as  } r \to \infty, \text{  for any nonnegative integer } k.
\ee
Furthermore, $M$ is smoothly asymptotic to a three dimensional cone at exponential rate \cite{Deruelle-2017}, i.e. there exist $R>0$, compact set $K$ in $M$, smooth closed surface $X$ and diffeomorphism $\phi:$ $M\setminus K$ $\rightarrow$ $C(X)\setminus \overline{B(o, R)}$ such that for any nonnegative integer $k$ %$\forall k \in \mathbb{N}$
\be\label{cone derivative estimates for metric}
|\na_C^k \big[(\phi^{-1})^*g-g_C\big]|_{g_C}(t,\omega)=O(t^{k-3}e^{-\frac{t^2}{4}}) \text{  as  } t \to \infty ;
\ee
\be\label{f in cone}
\frac{n}{2}-v\circ \phi^{-1}(t,\omega)=f\circ \phi^{-1}(t,\omega)=-\frac{t^2}{4}+c_0,
\ee
where $c_0$ is some constant. %In particular, $\Rm$ decays exponentially and $M$ has polynomial volume growth.
Hence the scalar curvature $S$ is integrable and nonnegative by Theorem \ref{sufficient condition for S geq 0}. Moreover, $(C(X), g_C)$ is Ricci flat \cite{Deruelle-2017} which implies that $X$ has constant Gauss curvature equal to $1$. By (\ref{f in cone}) or the proof of Theorem 3.2 in \cite{Deruelle-2017}, $X$ is diffeomorphic to the level sets $\{-f=s\}$ for all large $s$ (see also \cite{ChenDeruelle-2015}). Since $S\geq0$, using the results of Chen-Deruelle \cite{ChenDeruelle-2015} and Munteanu-Wang \cite{MunteanuWang-2012}, we see that $M$ is connected at infinity. Hence $X$ is connected.\\
\textbf{Case 1} $M$ is orientable.\\
The level sets of $f$ at infinity are orientable closed surfaces in $M$. Hence $X$ is also orientable and isometric to $\mathbb{S}^2(1)$. As a result, $C(X)$ is isometric to $\R^{3}\setminus \{0\}$. By (\ref{cone derivative estimates for Rm}) and (\ref{cone derivative estimates for metric}), $(M, g)$ is asymptotically flat and $\partial g_{ij}$ decay exponentially in $t=|x|$ with $x\in$ $\R^{3}$, in the standard coordinate of $\R^{3}\setminus \{0\}$, where $\partial$ denotes the partial derivative. The A.D.M. mass $m$ vanishes since
\be
m:=\frac{1}{32\pi}\lim_{t\to \infty} \int_{|x|=t}\sum^3_{i, j=1}\big(\frac{\partial g_{ij}}{\partial x_i}-\frac{\partial g_{ii}}{\partial x_j}\big)\frac{x_j}{t}d\sigma_t=0,
\ee
where $d\sigma_t$ is the volume element of the Euclidean sphere $\{|x|=t\}$ w.r.t. the Euclidean metric. By the rigidity case of the positive mass theorem, $M$ is isometric to $\R^3$ (see \cite{Schoen-1989}, \cite{Lohkamp-2016.0}, \cite{Lohkamp-2016}, \cite{SchoenYau-2017} and ref.therein).\\
\textbf{Case 2} $M$ is non-orientable.\\
We prove that it is impossible. Suppose on the contrary that $M$ is non-orientable. Then we consider $\pi: N \rightarrow M$, the orientable double cover of $M$. With $\pi^*g$ and $\pi^*f$, $N$ is endowed with the structure of a complete expanding Ricci soliton. Moreover, $S_N\geq 0$ and $\lim_{y\to\infty}r_N^2(y)S_N(y)=0$. By Case $1$, $N$ and thus $M$ are flat. $\na_g^2 f=-\frac{1}{2}g$ on $M$ and $M$ is diffeomorphic to $\R^3$, contradicting to the non-orientability of $M$.
\end{proof}

\section{Proof of Theorem \ref{non -ve scalar bdd preserved for conical data}}

We start with an estimate on the potential function $f$ under some growth conditions on $f$ and $S$.
\begin{lma}\label{quadratic estimates of v in exp under bdd S} Let $(M^n,g,f)$ be an $n$ dimensional complete noncompact gradient expanding Ricci soliton. If $f$ is proper, i.e. $\lim_{x\to \infty} f=-\infty$, and
\be\label{growth condition on S for the potential estimate}
\a:=\limsup_{x\to\infty} \frac{S}{v}<1,
\ee
where $v=\frac{n}{2}-f$, then for any $\delta$ $\in (0, 1-\a)$ there exists a positive constant $C$ such that
\be\label{quadratic estimate of potential the inequality}
\frac{\delta r^2}{4}-Cr-C\leq -f\leq \frac{r^2}{4}+Cr+C \text{  on  } M.
\ee
\end{lma}
\begin{remark}
 It can be seen from (\ref{general bdd for S in expander}) and (\ref{eqn of nav}) that when $f$ is proper, $\a$ is a real number in $[0,1]$.
\end{remark}
\begin{proof} The upper bound of (\ref{quadratic estimate of potential the inequality}) follows from (\ref{naf nav sqrtv control linearly in exp}) without any conditions on $f$ and $S$. For the lower bound, fix any $\delta$ $\in (0, 1-\a)$. We consider the flow $G_t$ of the vector field $\frac{\na v}{|\na v|^2}$ with $G_0$ be the identity map, where $v=\frac{n}{2}-f$. The flow makes sense since from (\ref{eqn of nav}) and (\ref{growth condition on S for the potential estimate}),
\begin{eqnarray}\notag
|\na v|^2 &=& v-\frac{n}{2}-S\\
\notag&=& v\Big(1-\frac{n}{2v}-\frac{S}{v}\Big)\\
\label{lower bdd for nav for estimates of -f}&\geq& \delta v\\
\notag&>& 1,
\end{eqnarray}
near infinity. Let $\rho$ be a large constant such that the above inequality holds on $\{v\geq \rho\}$ and $\{v= \rho\}$ is nonempty. It is not difficult to see that for all $q$ $\in\{v\geq \rho\}$, there are $t\geq 0$ and $z$ $\in\{v=\rho\}$  such that $G_t(z)=q$ and
\begin{eqnarray*}
v(q)-v(z)&=&\int_0^t\la\na v, \frac{\na v}{|\na v|^2}\ra(G_{\tau}(z)) d\tau\\
&=&t.
\end{eqnarray*}
Moreover, we have
\begin{eqnarray*}
d(q,z)&=&d(G_t(z), z)\\
&\leq& \int_0^t\frac{1}{|\na v|(G_{\tau}(z))}d\tau\\
&=& \int_0^t\frac{1}{\sqrt{v(G_{\tau}(z))-\frac{n}{2}-S}}  d\tau\\
&\leq& \int_0^t\frac{1}{\sqrt{\delta(\tau+\rho)}}  d\tau\\
&\leq& \frac{2}{\sqrt{\delta}}\sqrt{v(q)}.
\end{eqnarray*}
We used (\ref{lower bdd for nav for estimates of -f}) in the second last inequality. By triangular inequality,
\begin{eqnarray*}
r(q)&\leq& d(q,z) + \sup_{\{v=\rho\}}d(\cdot, p_0) \\
&\leq& \frac{2}{\sqrt{\delta}}\sqrt{v(q)}+ K_0,
\end{eqnarray*}
where $K_0=\sup_{\{v=\rho\}}d(\cdot, p_0)$.
Hence near infinity,
\be
v(q)\geq \frac{\delta r^2(q)}{4}-\frac{\delta K_0 r(q)}{2}+\frac{\delta K_0^2}{4}.
\ee
\end{proof}

Using the above lemmas, We prove a lower bound for the scalar curvature.
\begin{prop}\label{prop for norm sq Ric bdd by S and exp term}Let $(M^n, g, f)$ be a complete noncompact expanding gradient Ricci soliton with proper potential function $f$ and dimension $n\geq 2$. If
\be \label{S lower faster than quad}
\liminf_{x\to\infty} vS\geq 0,
\ee
then there exists a positive constant $C$ such that
\be\label{norm sq Ric bdd by S and exp term}
0\leq S+Cv^{1-\frac{n}{2}}e^{-v} \text{  on } M,
\ee
where $v=\frac{n}{2}-f$.
\end{prop}
\begin{proof}

%Most of the calculations here were done by Deruelle \cite{Deruelle-2017}, we include the details for the convenience of reader. Let $A$ be any constant great than or equal to $L_0+1$ and $u$ be $|\Ric|^2-AS$. By Lemma \ref{Ric -AS decay}
%\be \label{boundary condition of uv}
%\limsup_{x\to \infty}vu\leq 0.
%\ee

Let $u$ be $-S$. By (\ref{S lower faster than quad}), %and Lemma \ref{quadratic estimates of v in exp under bdd S}, $v\sim r^2$ and
\be \label{boundary condition of uv}
\limsup_{x\to \infty}vu\leq 0.
\ee

From (\ref{eqn of S}), we see that
\be\label{elementary eqn of u}
\D u\geq -u.
\ee
(\ref{norm sq Ric bdd by S and exp term}) then follows from (\ref{boundary condition of uv}), (\ref{elementary eqn of u}) and the proof of Lemma 2.9 in \cite{Deruelle-2017}. We include the details for the convenience of readers. The calculations were essentially done by Deruelle \cite{Deruelle-2017}. By the properness of $f$, $\lim_{x\to\infty}v=\infty$ and $v\geq 1$ outside some compact subset of $M$. We claim that for any constant $a\geq 4$, there exist positive constants $R_0$ and $b$ such that
\be\label{eqn for expav uv}
\Delta_{f+2\ln v-2av^{-1}}(e^{-\frac{a}{v}}uv)\geq 0 \text{  on  } \{u\geq 0\}\setminus B_{R_0}(p_0) \text{  and }
\ee
\be\label{e-bv v e-v is superharmonic}
\Delta_{f+2\ln v-2av^{-1}}(e^{-\frac{b}{v}}v^{2-\frac{n}{2}}e^{-v})<0 \text{  on  } \{u\geq 0\}\setminus B_{R_0}(p_0).
\ee
We first assume the above inequalities and prove (\ref{norm sq Ric bdd by S and exp term}).
%Let $R_0$ be a large positive number such that (\ref{eqn for expav uv}) and (\ref{e-bv v e-v is superharmonic}) hold.
There exists a positive constant $\beta$ such that
\be\label{inner bdry condition of Q}
e^{-\frac{a}{v}}uv-\beta e^{-\frac{b}{v}}v^{2-\frac{n}{2}}e^{-v}<0 \text{  on  } \partial B_{R_0}(p_0).
\ee
Let $Q:=e^{-\frac{a}{v}}uv-\beta e^{-\frac{b}{v}}v^{2-\frac{n}{2}}e^{-v}$. From $\lim_{x\to\infty}v=\infty$ and (\ref{boundary condition of uv}),
it is evident that $\limsup_{x\to\infty}Q\leq 0$.
For any $y\in$ $M\setminus \overline{B_{R_0}(p_0)}$ and $\varepsilon >0$, there exists $T>R_0$ such that $y\in$ $B_T(p_0)$ and
\be\label{outer bdry condition of Q}
Q<\varepsilon \text{  on  } \partial B_{T}(p_0).
\ee
We consider the open set $\Omega:= B_{T}(p_0)\setminus \overline{B_{R_0}(p_0)}$. Suppose $Q$ attains its maximum over $\overline{\Omega}$ at $z\in$ $\overline{\Omega}$. If $z\in$ $\partial\Omega$, then by (\ref{inner bdry condition of Q}) and (\ref{outer bdry condition of Q}), $Q(y)\leq Q(z)\leq \varepsilon$. If $z\in$ $\Omega$, we show that $Q(z)\leq 0$. Suppose on the contrary that $Q(z)>0$. Then we know that $u(z)\geq 0$ and by (\ref{eqn for expav uv}) and (\ref{e-bv v e-v is superharmonic}) at $z$,
\begin{eqnarray*}
0&\geq& \Delta_{f+2\ln v-2av^{-1}}Q\\
&=&\Delta_{f+2\ln v-2av^{-1}}(e^{-\frac{a}{v}}uv)-\beta\Delta_{f+2\ln v-2av^{-1}}(e^{-\frac{b}{v}}v^{2-\frac{n}{2}}e^{-v})\\
&\geq&-\beta\Delta_{f+2\ln v-2av^{-1}}(e^{-\frac{b}{v}}v^{2-\frac{n}{2}}e^{-v})\\
&>& 0,
\end{eqnarray*}
which is impossible. Hence, $Q(z)\leq 0$. In either cases, $Q(y)\leq Q(z)\leq \varepsilon$. We then obtain (\ref{norm sq Ric bdd by S and exp term}) by letting $\varepsilon\to 0$. It remains to prove (\ref{eqn for expav uv}) and (\ref{e-bv v e-v is superharmonic}). For (\ref{eqn for expav uv}), by (\ref{eqn of v}) and (\ref{elementary eqn of u}), we have

%by (\ref{eqn of Ric sq -AS}), we see that
%\be\label{elementary eqn of (Ric sq -AS)}
%\D u\geq -u.
%\ee
%Since $v\to\infty$ as $x \to \infty$, $v\geq 1$ outside some compact subset of $M$.
\begin{eqnarray*}
\D (uv)&=&v\D u+ u\D v +2\la \na u, \na v\ra\\
&\geq&-uv+uv+2\la \na \ln v, \na(vu)\ra -2|\na \ln v|^2uv;
\end{eqnarray*}
\be\label{partial ineq for uv in exp}
\Delta_{f+2\ln v} (uv)\geq-2|\na\ln v|^2uv.
\ee
For any $a>0$, we compute directly using (\ref{eqn of v}) as in Lemma 2.9 of \cite{Deruelle-2017} to get
\be\label{eqn for expav}
\Delta_{f+2\ln v}e^{-\frac{a}{v}}=\frac{a}{v}e^{-\frac{a}{v}}+\big(-\frac{4a}{v^3}+\frac{a^2}{v^4}\big)|\na v|^2e^{-\frac{a}{v}}.
\ee
A straightforward calculation using (\ref{partial ineq for uv in exp}) and (\ref{eqn for expav}) yields
\begin{eqnarray*}
\Delta_{f+2\ln v}(e^{-\frac{a}{v}}uv)&=&e^{-\frac{a}{v}}\Delta_{f+2\ln v} (uv)+uv\Delta_{f+2\ln v}e^{-\frac{a}{v}}+ 2\la\na e^{-\frac{a}{v}}, \na (uv) \ra\\
&\geq&-2|\na\ln v|^2e^{-\frac{a}{v}}uv+\frac{a}{v}e^{-\frac{a}{v}}uv+\big(-\frac{4a}{v^3}+\frac{a^2}{v^4}\big)|\na v|^2e^{-\frac{a}{v}}uv\\
&&+ 2\la\na e^{-\frac{a}{v}}, \na (uv) \ra\\
&=&\big(\frac{a}{v}-2|\na\ln v|^2\big)e^{-\frac{a}{v}}uv+\big(-\frac{4a}{v^3}+\frac{a^2}{v^4}\big)|\na v|^2e^{-\frac{a}{v}}uv\\
&&-2a\la \na v^{-1},\na (e^{-\frac{a}{v}}uv)\ra-\frac{2a^2}{v^4}|\na v|^2e^{-\frac{a}{v}}uv\\
&=&\big(\frac{a}{v}-2|\na\ln v|^2\big)e^{-\frac{a}{v}}uv-\big(\frac{4a}{v^3}+\frac{a^2}{v^4}\big)|\na v|^2e^{-\frac{a}{v}}uv\\
&&-2a\la \na v^{-1},\na (e^{-\frac{a}{v}}uv)\ra.
\end{eqnarray*}
Hence we have

$$\Delta_{f+2\ln v-2av^{-1}}(e^{-\frac{a}{v}}uv)\geq\big(\frac{a}{v}-2|\na\ln v|^2\big)e^{-\frac{a}{v}}uv-\big(\frac{4a}{v^3}+\frac{a^2}{v^4}\big)|\na v|^2e^{-\frac{a}{v}}uv.$$
On the set where $u$ is nonnegative and for $a\geq 4$, we may simplify the above inequality by $|\na v|^2\leq v$
%\begin{eqnarray}\notag
%\Delta_{f+2\ln v-2av^{-1}}(e^{-\frac{a}{v}}uv)&\geq&\frac{a-2}{v}e^{-\frac{a}{v}}uv-\big(\frac{4a}{v^2}+\frac{a^2}{v^3}\big)e^{-\frac{a}{v}}uv\\
%\label{eqn for expav uv}&\geq&\frac{1}{v}e^{-\frac{a}{v}}uv\\
%\notag &\geq& 0
%\end{eqnarray}
\begin{eqnarray*}
\Delta_{f+2\ln v-2av^{-1}}(e^{-\frac{a}{v}}uv)&\geq&\frac{a-2}{v}e^{-\frac{a}{v}}uv-\big(\frac{4a}{v^2}+\frac{a^2}{v^3}\big)e^{-\frac{a}{v}}uv\\
&\geq&\frac{1}{v}e^{-\frac{a}{v}}uv\\
&\geq& 0
\end{eqnarray*}
near infinity. We proved the inequality (\ref{eqn for expav uv}). For (\ref{e-bv v e-v is superharmonic}), by $|\na v|^2\leq v$, (\ref{eqn of nav}), (\ref{general bdd for S in expander}) and previous computation (\ref{half eq for e-v}),
\begin{eqnarray*}
\Delta_{f+2\ln v} (v^{2-\frac{n}{2}}e^{-v})&=& v^{2-\frac{n}{2}}e^{-v}\Big[ -S + (\frac{n}{2}-2)(\frac{n}{2}+1)|\na \ln v|^2\\
&&\qquad\qquad-(2S+n)(\frac{n}{2}-1)v^{-1}\Big].\\
&\leq& v^{2-\frac{n}{2}}e^{-v}\Big[ -S + |\frac{n}{2}-2|(\frac{n}{2}+1)v^{-1}\Big].
\end{eqnarray*}
\begin{eqnarray*}
2a\la \na v^{-1}, \na  (v^{2-\frac{n}{2}}e^{-v})\ra &=& a(n-4)\frac{|\na \ln v|^2}{v}v^{2-\frac{n}{2}}e^{-v}+2a|\na \ln v|^2v^{2-\frac{n}{2}}e^{-v}\\
&\leq&v^{2-\frac{n}{2}}e^{-v}\Big[\frac{a|n-4|}{v^{2}}+\frac{2a}{v}\Big].
\end{eqnarray*}
Hence
\begin{eqnarray*}
\Delta_{f+2\ln v-2av^{-1}} (v^{2-\frac{n}{2}}e^{-v})&\leq&v^{2-\frac{n}{2}}e^{-v}\Big[-S+\frac{C_0}{v}\Big],
\end{eqnarray*}
where $C_0=C_0(a, n)$ is a positive constant depending on $a$ and $n$. Using (\ref{S lower faster than quad}), we have $vS\geq -1$ near infinity. We can rewrite the above inequality as
\be\label{ineq for ve-v}
\Delta_{f+2\ln v-2av^{-1}}(v^{2-\frac{n}{2}}e^{-v})\leq C_1v^{1-\frac{n}{2}}e^{-v},
\ee
where $C_1=C_1(a, n)$ is a positive constant depending on $a$ and $n$. For any $b>0$,
\be
2a\la \na v^{-1}, \na e^{-\frac{b}{v}}\ra=-2ab|\na \ln v|^2v^{-2}e^{-\frac{b}{v}}.
\ee
By (\ref{eqn for expav}) and $|\na v|^2\leq v$,
\begin{eqnarray}\notag
\Delta_{f+2\ln v-2av^{-1}}e^{-\frac{b}{v}}&=&e^{-\frac{b}{v}}\Big[\frac{b}{v}+\Big(-\frac{4b}{v^3}+\frac{b^2}{v^4}\Big)|\na v|^2-2ab|\na\ln v|^2v^{-2}\Big]\\
\label{ineq for e-bv}&\leq&e^{-\frac{b}{v}}\Big[\frac{b}{v}+\frac{b^2}{v^4}|\na v|^2\Big]\\
\notag&\leq&e^{-\frac{b}{v}}\Big[\frac{b}{v}+\frac{b^2}{v^3}\Big].
\end{eqnarray}
It follows from (\ref{eqn of nav}) that
\begin{eqnarray}\notag
2\la \na(v^{2-\frac{n}{2}}e^{-v}), \na e^{-\frac{b}{v}}\ra &=& 2bv^{-2}e^{-\frac{b}{v}}|\na v|^2(2-\frac{n}{2})v^{1-\frac{n}{2}}e^{-v}\\
\notag&&+2bv^{-2}e^{-\frac{b}{v}}|\na v|^2(-e^{-v})v^{2-\frac{n}{2}}\\
\label{cross term v-ve e-bv}&=&2bv^{-2}e^{-\frac{b}{v}}|\na v|^2(2-\frac{n}{2})v^{1-\frac{n}{2}}e^{-v}\\
\notag&&-\frac{2b}{v}e^{-\frac{b}{v}}e^{-v}v^{2-\frac{n}{2}}+\frac{b(2S+n)}{v^2}e^{-\frac{b}{v}}e^{-v}v^{2-\frac{n}{2}}\\
\notag&\leq&-\frac{2b}{v}e^{-\frac{b}{v}}e^{-v}v^{2-\frac{n}{2}}+\frac{b(2S+n+|n-4|)}{v^2}e^{-\frac{b}{v}}e^{-v}v^{2-\frac{n}{2}}.
\end{eqnarray}
Using (\ref{ineq for ve-v}), (\ref{ineq for e-bv}) and (\ref{cross term v-ve e-bv}), on the set where $u$ is nonnegative (i.e. $S\leq 0$),
\begin{eqnarray*}
\Delta_{f+2\ln v-2av^{-1}}(e^{-\frac{b}{v}}v^{2-\frac{n}{2}}e^{-v})&=&e^{-\frac{b}{v}}\Delta_{f+2\ln v-2av^{-1}}(v^{2-\frac{n}{2}}e^{-v})\\
&&+v^{2-\frac{n}{2}}e^{-v}\Delta_{f+2\ln v-2av^{-1}}e^{-\frac{b}{v}}\\
&&+2\la \na(v^{2-\frac{n}{2}}e^{-v}), \na e^{-\frac{b}{v}}\ra\\
&\leq&\frac{C_1}{v}e^{-\frac{b}{v}}v^{2-\frac{n}{2}}e^{-v}+\Big[\frac{b}{v}+\frac{b^2}{v^3}\Big]e^{-\frac{b}{v}}v^{2-\frac{n}{2}}e^{-v}\\
&&-\frac{2b}{v}e^{-\frac{b}{v}}e^{-v}v^{2-\frac{n}{2}}+\frac{b(2S+n+|n-4|)}{v^2}e^{-\frac{b}{v}}e^{-v}v^{2-\frac{n}{2}}\\
&\leq&e^{-\frac{b}{v}}v^{2-\frac{n}{2}}e^{-v}\Big[\frac{C_1-b}{v}+\frac{C_2}{v^2}\Big],
\end{eqnarray*}
where $C_2=C_2(b, n)$ is positive and depends on $b$ and $n$. We choose $b=C_1+2$, then
%\begin{eqnarray}\notag
%\Delta_{f+2\ln v-2av^{-1}}(e^{-\frac{b}{v}}v^{2-\frac{n}{2}}e^{-v})&\leq& e^{-\frac{b}{v}}v^{2-\frac{n}{2}}e^{-v}\Big[-\frac{2}{v}+\frac{C_2}{v^2}\Big]\\
%\label{e-bv v e-v is superharmonic}&\leq&-\frac{1}{v}e^{-\frac{b}{v}}v^{2-\frac{n}{2}}e^{-v}\\
%\notag&<&0
%\end{eqnarray}

\begin{eqnarray*}
\Delta_{f+2\ln v-2av^{-1}}(e^{-\frac{b}{v}}v^{2-\frac{n}{2}}e^{-v})&\leq& e^{-\frac{b}{v}}v^{2-\frac{n}{2}}e^{-v}\Big[-\frac{2}{v}+\frac{C_2}{v^2}\Big]\\
&\leq&-\frac{1}{v}e^{-\frac{b}{v}}v^{2-\frac{n}{2}}e^{-v}\\
&<&0
\end{eqnarray*}

near infinity. We showed (\ref{e-bv v e-v is superharmonic}) and completed the proof of the proposition.
\end{proof}

With the above preparation, we give a proof for Theorem \ref{non -ve scalar bdd preserved for conical data}.

\begin{proof}[\textbf{Proof of Theorem \ref{non -ve scalar bdd preserved for conical data}:}] From (\ref{naf nav sqrtv control linearly in exp}), we see that (\ref{S lower faster than quad}) holds. By Proposition \ref{prop for norm sq Ric bdd by S and exp term}, the negative part of the scalar curvature $S_-$ satisfies
$$0\leq S_-\leq Cv^{1-\frac{n}{2}}e^{-v} \text{  on  } M,$$
for some constant $C>0$. Using Lemma \ref{quadratic estimates of v in exp under bdd S}, we have $v\sim r^2$. Hence there exists $\delta>0$ such that
\be
S_-\leq Ce^{-\delta r^2}.
\ee
By (\ref{naf nav sqrtv control linearly in exp}) and Theorem $1.1$ in \cite{MunteanuWang-2014}, there exists a positive constant $C$ such that for all $R>0$
$$Vol_g(B_R(p_0))\leq Ce^{\sqrt{n-1}R},$$
where $Vol_g(B_R(p_0))$ is the volume of the geodesic ball $B_R(p_0)$. It implies that
$$\int_M S_-<\infty.$$
The nonnegativity of the scalar curvature is a consequence of $\liminf_{x\to \infty} r^2S\geq0$ and Theorem \ref{sufficient condition for S geq 0}. Connectedness at infinity of $M$ now follows from \cite{ChenDeruelle-2015} and  \cite{MunteanuWang-2012}.
\end{proof}

\section{Proof of theorem \ref{bdd S and f proper implies bdd Rm}}
It was observed by Munteanu-Wang \cite{MunteanuWang-2015.3} that in four dimensional gradient Ricci soliton, the Riemann curvature $Rm$ can be bounded by $\Ric$ and $\na \Ric$:
\begin{lma}\label{Rm to Rc}\cite{MunteanuWang-2015.3} Let $(M^4,g,f)$ be a four dimensional gradient Ricci soliton. There exists a universal positive constant $A_0$ such that if $\na f \neq 0$ at $q$ $\in M$, then at $q$
\be\label{Rm controlled by Rc}
|\Rm|\leq A_0\Big(|\Ric|+\frac{|\na \Ric|}{|\na f|}\Big).
\ee
\end{lma}
Please see \cite{MunteanuWang-2015.3} and \cite{Chan-2019} for a proof of Lemma \ref{Rm to Rc}. To prove that the curvature tensor is bounded, we first give estimate on the Ricci tensor.
\begin{thm} \label{bdd S and f proper implies bdd Ric}Let $(M^4, g, f)$ be a $4$ dimensional complete noncompact expanding gradient Ricci soliton. Suppose that it has bounded scalar curvature and $f\to -\infty$ as $x\to \infty$. Then the Ricci tensor $\Ric$ is bounded.
\end{thm}
\begin{proof}
Let $L:=\sup_M |S|<\infty$. If $L=0$, then by (\ref{eqn of S}) $M$ is Ricci flat and we are done with the proof. We may assume $L>0$ and introduce the function $F(S):=(S+3L)^{-a}>0$, where $a>0$. We see that
\be\label{C0 estimate of F}
\frac{1}{4^aL^a}\leq F\leq \frac{1}{2^aL^a},
\ee
and
\be\label{eqn for na F}
\na F=-a(S+3L)^{-a-1}\na S.
\ee
By (\ref{eqn of S})
\begin{eqnarray}\notag
\D F&=&-a(S+3L)^{-a-1}\D S + a(a+1)(S+3L)^{-a-2}|\na S|^2\\
\label{eqn for lap F}&=&a(S+3L)^{-a-1}(S+2|\Ric|^2)+a(a+1)(S+3L)^{-a-2}|\na S|^2.
\end{eqnarray}
Whenever $\na f\neq 0$, we compute directly using (\ref{eqn of Rc in exp}) and get
\begin{eqnarray}
\notag\D |\Ric|^2&=& 2|\na \Ric|^2+2\la \Ric, \D\Ric\ra\\
\notag&\geq&2|\na \Ric|^2-2|\Ric|^2-4|\Ric|^2|\Rm|\\
\label{ineq for Ric norm sq in dim 4}&\geq&2|\na \Ric|^2-2|\Ric|^2-4A_0|\Ric|^3-4A_0|\Ric|^2\frac{|\na \Ric|}{|\na f|}\\
\notag&\geq&|\na \Ric|^2-2|\Ric|^2-4A_0|\Ric|^3-\frac{4A_0^2|\Ric|^4}{|\na f|^2}.
\end{eqnarray}
It is not difficult to see from Cauchy Schwarz inequality, (\ref{eqn for lap F}), (\ref{ineq for Ric norm sq in dim 4}) and (\ref{C0 estimate of F}) that
\begin{eqnarray*}
\D (F |\Ric|^2)&=& F\D |\Ric|^2+2\la \na F, \na |\Ric|^2\ra +|\Ric|^2\D F\\
&\geq&F\D |\Ric|^2-4|F'||\na S||\Ric||\na\Ric|+|\Ric|^2\D F\\
&\geq&F\D |\Ric|^2-F|\na \Ric|^2-4\frac{(F')^2}{F}|\na S|^2|\Ric|^2\\
&&+|\Ric|^2\D F\\
&\geq&F\D |\Ric|^2-F|\na \Ric|^2-4\frac{(F')^2}{F}|\na S|^2|\Ric|^2\\
&&+a(S+3L)^{-a-1}(S+2|\Ric|^2)|\Ric|^2\\
&&+a(a+1)(S+3L)^{-a-2}|\na S|^2|\Ric|^2\\
&=&F\D |\Ric|^2-F|\na \Ric|^2-4a^2(S+3L)^{-a-2}|\na S|^2|\Ric|^2\\
&&+a(S+3L)^{-a-1}(S+2|\Ric|^2)|\Ric|^2\\
&&+a(a+1)(S+3L)^{-a-2}|\na S|^2|\Ric|^2\\
&\geq&-\big(2+\frac{a}{2}\big)F|\Ric|^2-4A_0F|\Ric|^3-\frac{4A_0^2F|\Ric|^4}{|\na f|^2}\\
&&+2a(S+3L)^{a-1}F^2|\Ric|^4\\
&&+a(1-3a)(S+3L)^{-a-2}|\na S|^2|\Ric|^2.
\end{eqnarray*}
We may take $a=\frac{1}{3}$ and consider the function $w:=F|\Ric|^2$, then the above inequalities can be rewritten as
\begin{eqnarray*}
\D w\geq \Big(\frac{2}{3}F^2-\frac{4A_0^2}{F|\na f|^2}\Big)w^2-4A_0F^{-\frac{1}{2}}w^{\frac{3}{2}}-\frac{13}{6}w.
\end{eqnarray*}
By (\ref{eqn of naf}), $|\na f|^2=-f-S\geq -f-L\to \infty$. Using (\ref{C0 estimate of F}), we have outside a compact set,
\begin{eqnarray*}
\D w&\geq& \Big(\frac{1}{2^{\frac{1}{3}}3L^{\frac{2}{3}}}-\frac{4^{\frac{4}{3}}A_0^2L^{\frac{1}{3}}}{|\na f|^2}\Big)w^2-4A_0F^{-\frac{1}{2}}w^{\frac{3}{2}}-\frac{13}{6}w\\
&\geq&\frac{1}{6L^{\frac{2}{3}}}w^2-4A_0F^{-\frac{1}{2}}w^{\frac{3}{2}}-\frac{13}{6}w.
\end{eqnarray*}
It is easy to see that the differential inequality of $w$ resembles the one of $u$ in (\ref{reduced ineq for norm of RC sq in exp}). Hence a similar cut off function argument as in the proof of Theorem \ref{curv estimate in dim 3 exp}\ref{3 dim S control Rm exp} and the Laplacian comparison theorem (\ref{lap comp in exp}) (see \cite{WeiWylie-2009}) yields the boundedness of $w$.
\end{proof}
Theorem \ref{bdd S and f proper implies bdd Rm} is now a consequence of the following theorem. The proof of the theorem is essentially due to Munteanu-Wang \cite{MunteanuWang-2015.3} and Cao-Cui \cite{CaoCui-2014}. We shall only give a sketch of the proof.
\begin{thm} \label{bdd Ric and f proper implies bdd Rm}Let $(M^4, g, f)$ be a $4$ dimensional complete noncompact expanding gradient Ricci soliton. Suppose that it has bounded Ricci curvature and $f\to -\infty$ as $x\to \infty$. Then the curvature tensor $\Rm$ is bounded
\end{thm}
\begin{proof} Throughout this proof, we use $c_0$ to denote some constants depending on the global upper bound of $|\Ric|$ and $A_0$ in Lemma \ref{Rm to Rc}, its value may be different from line by line. Since $|\na f|^2=-f-S\to \infty$ as $x\to \infty$, we have by Lemma \ref{Rm to Rc}
\begin{eqnarray*}
|\na \Ric|^2&\geq&\frac{1}{2A_0^2}|\Rm|^2-|\Ric|^2\\
&\geq&\frac{1}{2A_0^2}|\Rm|^2-c_0
\end{eqnarray*}
outside a compact subset of $M$. Moreover by (\ref{ineq for Ric norm sq in dim 4})
\begin{eqnarray}
\notag\D |\Ric|^2&\geq& |\na\Ric|^2 -c_0\\
\notag&\geq&\frac{1}{2A_0^2}|\Rm|^2-c_0\\
\label{ineq for norm sq Ric under Ric bdd in dim 4} &\geq&\frac{1}{4A_0^2}(|\Rm|+\lambda|\Ric|^2 )^2-\frac{\lambda^2}{2A_0^2}|\Ric|^4-c_0\\
\notag&\geq&\frac{1}{4A_0^2}(|\Rm|+\lambda|\Ric|^2 )^2-\frac{\lambda^2}{2A_0^2}c_0-c_0,
\end{eqnarray}
where $\lambda$ is any non-negative constant. We apply (\ref{eqn of Rm in exp}) and get
\begin{eqnarray*}
\D |\Rm|^2&=&2|\na \Rm|^2+2\langle\Rm,\D \Rm\rangle\\
&\geq&2|\na \Rm|^2-2|\Rm|^2-c|\Rm|^3,
\end{eqnarray*}
where $c$ is an absolute constant. By Kato's inequality,
$$\D |\Rm|\geq -|\Rm|-c|\Rm|^2$$
on the set where $|\Rm|>0$. Using (\ref{ineq for norm sq Ric under Ric bdd in dim 4}), we see that for all $\lambda\geq 0$,

$$\D \big(|\Rm|+\lambda|\Ric|^2\big)\geq \big(\frac{\lambda}{4A_0^2}-c\big)\big(|\Rm|+\lambda|\Ric|^2\big)^2-|\Rm|-\frac{\lambda^3}{2A_0^2}c_0-c_0\lambda.$$
Let $W:=|\Rm|+\lambda|\Ric|^2$, we may rewrite the above inequality as
\begin{eqnarray*}
\D W&\geq& (\frac{\lambda}{4A_0^2}-c\big)W^2-W-\frac{\lambda^3}{2A_0^2}c_0-c_0\lambda\\
&\geq& W^2-W-\frac{\lambda^3}{2A_0^2}c_0-c_0\lambda
\end{eqnarray*}
for all sufficiently large $\lambda$. Similar argument as in the proofs of Theorems \ref{curv estimate in dim 3 exp}\ref{3 dim S control Rm exp} and \ref{bdd S and f proper implies bdd Ric} gives the boundedness of $W$.
\end{proof}

\end{document}